\PassOptionsToPackage{table}{xcolor}

\documentclass[UKenglish,letterpaper,11pt]{article}

\usepackage{babel}
\usepackage[utf8]{inputenc}
\usepackage[margin=1in]{geometry}
\usepackage{color}
\usepackage{graphicx}
\usepackage{xspace}
\usepackage{url}
\usepackage{amsmath,amssymb}
\usepackage{amsthm}
\usepackage{thmtools}
\usepackage{hyperref}
\usepackage[capitalize]{cleveref}
\usepackage{authblk}
\usepackage{todonotes}
\usepackage{paralist}
\usepackage{booktabs}
\usepackage{bbding}

\usepackage{bm}

\usepackage{tikz}
\tikzstyle vertex=[circle, draw, fill=black, inner sep=0.5mm]

\usetikzlibrary{arrows,decorations.markings}
\usetikzlibrary{arrows.meta}

\newcommand{\tww}{\text{tww}\xspace}
\newcommand{\sep}{\text{sep}\xspace}

\newcommand{\perm}{\mathfrak{S}}
\newcommand{\Sep}{\mathcal{S}}
\newcommand{\Av}{\text{Av}}

\newcommand{\Ac}{\mathcal{A}}
\newcommand{\Cc}{\mathcal{C}}
\newcommand{\Dc}{\mathcal{D}}
\newcommand{\Ic}{\mathcal{I}}

\newcommand{\Pc}{\mathcal{P}}
\newcommand{\Rc}{\mathcal{R}}

\newcommand{\Nn}{\mathbb{N}}

\newcommand{\card}[1]{\left|#1\right|}

\newcommand{\ceil}[1]{\left\lceil #1 \right\rceil}
\newcommand{\clos}[1]{\overline{#1}}    

\newcommand{\first}{\mathbf{first}}
\newcommand{\last}{\mathbf{last}}
\newcommand{\arity}{ar}

\newtheorem{theorem}{Theorem}[section]

\newtheorem{lemma}[theorem]{Lemma}
\newtheorem{corollary}[theorem]{Corollary}

\newtheorem{claim}[theorem]{Claim}
\newtheorem{remark}[theorem]{Remark}

\theoremstyle{definition}

\newtheorem{question}[theorem]{Question}

\Crefname{figure}{Figure}{Figures}
\Crefmultiformat{claim}{Claims~#2#1#3}{ and~#2#1#3}{, #2#1#3}{, and~#2#1#3}

\newenvironment{claimproof}{\begin{proof}}{\end{proof}}

\title{Factoring Pattern-Free Permutations into Separable ones}

\author{Édouard Bonnet}
\author{Romain Bourneuf}
\author{Colin Geniet}
\author{Stéphan Thomassé}
\affil{Univ.\ Lyon, ENS de Lyon, UCBL, CNRS, LIP, France}

\date{}

\begin{document}

\maketitle

\begin{abstract}
We show that for any permutation $\pi$ there exists an integer $k_{\pi}$ such that every permutation avoiding $\pi$ as a~pattern is a~product of at most $k_{\pi}$ separable permutations.
In other words, every strict class $\cal C$ of permutations is contained in a~bounded power of the class of separable permutations.
This factorisation can be computed in linear time, for any fixed $\pi$.

The central tool for our result is a~notion of width of permutations,
introduced by Guillemot and Marx [SODA '14] to efficiently detect patterns,
and later generalised to graphs and matrices under the name of twin-width.
Specifically, our factorisation is inspired by the decomposition used
in the recent result that graphs with bounded twin-width are polynomially $\chi$-bounded.
As an application, we show that there is a~fixed class $\Cc$ of graphs of bounded twin-width
such that every class of bounded twin-width is a~first-order transduction of $\Cc$.
\end{abstract}

\section{Introduction}\label{sec:intro}

Given a~class $\cal C$ of discrete structures, the arguably preeminent algorithmic task is the one of recognition: does the input belong to $\cal C$? 
This problem is often tied with an effective \emph{construction} of the class $\cal C$ by performing elementary operations on some basic building blocks. 
For instance, totally unimodular matrices~\cite{Seymour80a}, minor-closed classes~\cite{RobertsonS03a}, and perfect graphs~\cite{chudnovsky2006strong} can be constructed from simpler objects, respectively, network matrices, graphs embeddable on low-genus surfaces, and variants of bipartite graphs. 
In this paper we show that \emph{strict} classes of permutations, that is, those avoiding a~fixed pattern, can be constructed from separable permutations (the basic class) via~a bounded number of compositions (the elementary operation).

Given a~positive integer $n$, we denote by $[n]$ the set $\{1, 2, \ldots, n\}$.
Let $n\le m$ be two integers, we say that a~permutation $\pi\in \perm_n$ is a~\emph{pattern} of $\sigma \in \perm_m$ if there is an increasing function $f$ from~$[n]$ to $[m]$ such that $\pi(i)<\pi(j)$ if and only if $\sigma(f(i))<\sigma(f(j))$ for all $i,j \in [n]$. 
Another way of characterizing patterns is to associate to a~permutation $\sigma\in \perm_n$ its $n\times n$ matrix $A(\sigma)=(a_{ij})$ with $a_{ij}=1$ if $j=\sigma(i)$, and $a_{ij}=0$ otherwise. 
Observe that $\pi$ is a~pattern of $\sigma$ if and only if $A(\pi)$ is a~submatrix of $A(\sigma)$. 
For instance, 12345 is a~pattern of $\sigma$ if it contains an increasing subsequence of length five. 
A crucial achievement in permutation patterns is the Guillemot--Marx algorithm, which decides if a~permutation~$\pi$ is a~pattern of $\sigma$ in time $f(\pi) \cdot |\sigma|$, where $|\sigma|$ is the size of $\sigma$.

Patterns readily offer a~complexity notion for permutations: A~permutation is ``simple" if it does not contain a~fixed small pattern. 
We will consider \emph{classes} of permutations, which are assumed closed under taking patterns. 
The existence of a~gap between the class of all permutations and any strict class is illustrated by the Marcus--Tardos theorem, answering the Stanley--Wilf conjecture: Every strict class of permutations has at most $2^{O(n)}$ permutations of size $n$, whereas the class of all permutations obviously has $n!=2^{\Theta(n \log n)}$ such permutations.
From an algorithmic perspective, sequences avoiding a~fixed pattern can be comparison-sorted in almost linear time $O\left(n \cdot 2^{(1 + o(1)) \alpha(n)}\right)$ where $\alpha$ is the inverse Ackermann function~\cite{chalermsook2015bst,kozma2020heaps,chalermsook2023sorting},
while linear algorithms when excluding some specific small patterns have been long known~\cite{knuth1968art,arthur2007sorting}.
Furthermore, as a~generalisation of Guillemot--Marx algorithm,
any property defined using first-order logic (FO) can be tested inside any strict permutation class~$\Cc$ in linear time~\cite{twin-width1}. For instance, given a~fixed permutation~$\tau$, one can decide in linear time if an input $n$-permutation~$\sigma$ in~$\Cc$ is such that every pair of elements $i,j\in [n]$ is contained in a~pattern~$\tau$ of~$\sigma$. 
Observe that even the existence of a~linear-size positive certificate for this seemingly quadratic problem is far from obvious.

Strict classes of permutations are thus significantly simpler, both algorithmically and in terms of growth.
The next question is to construct them from a~basic class using some simple operations. 
In the case of permutations, possibly the most natural elementary operation is the product (or composition).
Furthermore, the class of \emph{separable permutations} is basic in several ways.
It consists of those permutations whose permutation graph is a~cograph; an elementary graph class (which coincides with graphs of twin-width~0).
Like cographs have a~natural auxiliary tree structure (the \emph{cotrees}), separable permutations inherit their own tree structure, the so-called \emph{separating tree}~\cite{BoseBL98}.
Separable permutations are originally themselves defined from the trivial permutation~1, by successively applying direct sums (setting two permutation matrices as diagonal blocks of the new permutation matrix) or skew sums (the same with antidiagonal blocks), or equivalently by closing $\{12,21\}$ under substitution.
They are well known to be the permutations avoiding the patterns 2413 and 3142~\cite{BoseBL98}.

As our main result, we show:
\begin{restatable}{theorem}{mainthm}\label{thm:main}
    For any pattern~$\pi$, there exists $k_\pi = 2^{2^{O(\card{\pi})}}$
    such that every permutation avoiding~$\pi$ is a~product of at most $k_\pi$ separable permutations.
\end{restatable}
Combining~\cref{thm:main} with the definition of separable permutations, every permutation of~\emph{$\Av(\pi)$}, the set of permutations avoiding the pattern $\pi$, can be built from the trivial permutation~1 via direct and skew sums, followed by a~bounded-length product.
Conversely, remark that for any~$c$, the class of products of~$c$ separable permutations
avoids some pattern, since it contains only~$2^{O(cn)}$ permutations on~$n$ elements.

The proof of \cref{thm:main} is effective,
and yields a~fixed-parameter tractable (FPT) algorithm to compute the factorisation.
With some more work, we show how to implement it in linear time.

\subsection{Shortest separable decompositions}
We call \emph{separable index} of a~permutation~$\sigma$, denoted~$\sep(\sigma)$,
the smallest~$k$ such that~$\sigma$ is the product of~$k$ separable permutations.
Our main result then states that permutations avoiding a~pattern have bounded separable index.
Our upper bound is doubly exponential: we show that permutations avoiding a~pattern~$\pi$ of length~$k$
have separable index at most~$2^{2^{O(k)}}$.
It is natural to ask to which extent this bound can be improved.

For specific patterns~$\pi$, this bound may be very low:
if~$\pi$ is the decreasing permutation of length~$k$,
then it is relatively simple to show that permutations avoiding~$\pi$ have separable index~$O(\log k)$,
see \cref{lem:shuffle-decomp} below.
However this situation is unusual:
from lower bounds on the growth of pattern-avoiding classes due to Fox~\cite{fox2013stanleywilf},
it follows that for almost all patterns~$\pi$ of size~$k$,
the maximum separable index of permutations avoiding~$\pi$ is at~least polynomial, more
precisely at~least~$\Omega(k^{1/4 - \epsilon})$ for any $\epsilon > 0$.
This still leaves a~huge gap between our double-exponential upper bound and the polynomial lower bound.
In principle the following question could be met with a~positive answer.

\begin{question}
Is the maximum separable index among permutations avoiding a~pattern of size $k$ polynomial in~$k$?
\end{question}
As an easier first step, one could try and bring the upper bound down to singly exponential.

From a~different perspective, one may consider
the problem of computing~$\sep(\sigma)$, given some permutation~$\sigma$.
Our result implies an FPT approximation.
Given~$k \in \Nn$, there is a~pattern~$\pi$ that every product of~$k$ separable permutations avoids.
Conversely, by our result, any permutation avoiding~$\pi$
is product of at most~$f(k)$ separable permutations for some function~$f$.
Applying the pattern recognition algorithm of Guillemot and Marx~\cite{guillemot14patterns},
we can either detect that~$\sigma$ contains~$\pi$ and conclude that~$\sep(\sigma) > k$,
or find that~$\sigma$ avoids~$\pi$ and conclude that~$\sep(\sigma) \le f(k)$.
Here, the size of~$\pi$ is exponential in~$k$,
hence the approximation function~$f$ is triple exponential in~$k$.
We ask whether this can be improved, possibly up to an exact algorithm:
\begin{question}
    What is the parameterised complexity of computing~$\sep(\sigma)$,
    and of approximating it within a~constant factor?
\end{question}

\subsection{Applications to graphs}
At a~high level, our result states that permutations avoiding~$\pi$
can be decomposed into permutations avoiding a~fixed pattern~$\tau$,
where the length of the decomposition depends on~$\pi$, but~$\tau$ does not.
Trading patterns with twin-width, we get results on graphs
that convey a~similar flavour of transformation into structures with universally bounded twin-width.
For sparse graphs, this transformation is the $d$-\emph{subdivision},
which consists of replacing every edge in a~given graph by a~path of length~$d$.
\begin{restatable}{theorem}{twwsubdivision}
  \label{thm:sparsetww-subdivision}
  There is a~universal constant~$c$ and a~function~$f$ such that
  for any graph~$G$ with twin-width~$k$ and no~$K_{t,t}$-subgraph,
  the $f(k,t)$-subdivision of~$G$ has twin-width at most~$c$.
\end{restatable}

This can be extended to dense graphs (and more generally binary relational structures)
using \emph{first-order (FO) transductions},
which roughly speaking are transformations that can be described using first-order logic.
\begin{restatable}{theorem}{FOdesc}
  \label{thm:tww-FO-desc}
  There is a~fixed strict class~$\Cc$ of permutations
  such that for any class~$\Dc$ of binary structures, $\Dc$ has bounded twin-width iff
  there exists an FO transduction~$\Phi$ satisfying~$\Dc \subseteq \Phi(\Cc)$.
\end{restatable}
We then say that \emph{$\Dc$ is an FO transduction of $\Cc$}.
\cref{thm:tww-FO-desc} nicely echoes similar characterizations for linear clique-width and clique-width, namely the fact that a~graph class has bounded linear clique-width if and only if it is an FO transduction of a~linear order, and bounded clique-width if and only if it is an FO transduction of a~tree order~\cite{Colcombet07}.
Fittingly, these results can also be phrased in terms of pattern-avoiding permutation classes.
Indeed a~class has bounded linear clique-width if and only if it is an FO transduction of $\Av(21)$, and bounded clique-width if and only if it is an FO transduction of $\Av(231)$, or equivalently an FO transduction of the class of separable permutations (see for example~\cite[Proposition 8.1]{tww-perm}).

Reconstructing a~graph~$G$ from the $d$-subdivision of~$G$ with fixed~$d$ is a~simple example of FO transduction,
hence \cref{thm:sparsetww-subdivision} is a~special case of \cref{thm:tww-FO-desc} utilizing a~specific transduction.
\Cref{thm:sparsetww-subdivision,thm:tww-FO-desc} are based on a~representation of
any factorisation $\sigma = \sigma_m \circ \dots \circ \sigma_1$ as a~system of paths between ordered sets of vertices:
each permutation~$\sigma_i$ is represented by a~matching between two orders,
which are joined together into paths of length~$m$, see \cref{fig:path-system} for an example.
\begin{figure}
    \begin{center}
        \begin{tikzpicture}
            \def\h{0.6}
            \foreach \i in {0,...,3}{
                \foreach \x in {1,...,6}{
                    \node[vertex] (\i\x) at (2*\i,\h*\x) {};
                }
                \node (X\i) at (2*\i,0) {$X_\i$};
            }

            \foreach \x/\sx in {1/1,2/6,3/3,4/5,5/2,6/4}{
                \node (\sx) at (-0.5,\h*\x) {$\sx$};
                \node (s\x) at (6.5,\h*\x) {$\x$};
            }

            \foreach \x/\sx in {1/1,2/2,3/3,4/5,5/4,6/6}{
                \draw (0\x) -- (1\sx);
            }
            \foreach \x/\sx in {1/1,2/4,3/5,4/2,5/3,6/6}{
                \draw (1\x) -- (2\sx);
            }
            \foreach \x/\sx in {1/1,2/2,3/5,4/6,5/3,6/4}{
                \draw (2\x) -- (3\sx);
            }
        \end{tikzpicture}
    \end{center}
    \caption{A factorisation of~163524 into three separable permutations, represented as a~path system.}
    \label{fig:path-system}
\end{figure}
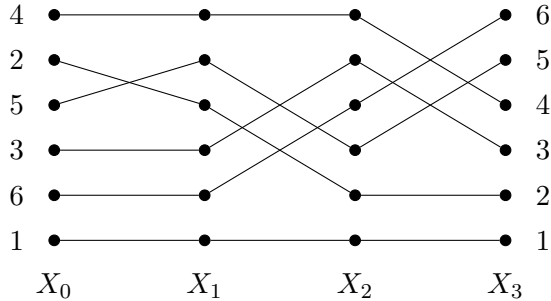
When the permutations~$\sigma_i$ are separable (or more generally avoid any fixed pattern),
these path system representations have bounded twin-width, independently of the length of the factorisation.
Furthermore, for any fixed length~$m$, there is an FO transduction depending only on~$m$
that can reconstruct~$\sigma$ from the path system representation.
We obtain \cref{thm:tww-FO-desc} by combining this with the following result of Bonnet et al.
\begin{theorem}[\cite{tww-perm}]
    \label{thm:FO-desc2}
    For any class~$\Dc$ of binary structures with bounded twin-width,
    there exists a~pattern avoiding class of permutations~$\Cc$
    and an FO transduction~$\Phi$ such that~$\Dc \subseteq \Phi(\Cc)$.
\end{theorem}
Notice that the difference with \cref{thm:tww-FO-desc} is that the class~$\Cc$ depends on~$\Dc$.

\subsection{Overview of the proof of \cref{thm:main}}

We see an $n$-element permutation $\sigma$ as two linear orders $<,\prec$ over $[n]$. 

\paragraph*{Mixed minor free permutations.}
We fix a~permutation $\pi$, and consider permutations $\sigma$ in $\Av(\pi)$, the class of all permutations avoiding $\pi$ as a~pattern.
This class has bounded twin-width~\cite{guillemot14patterns,twin-width1}, and it was further proven that there is a~third linear order $<_{\tww}$ over $[n]$ such that the adjacency matrix\footnote{At this point, we do not wish to concern the reader with the actual encoding.} of the binary structure $([n],<,\prec)$ has no $c_\pi$-mixed minor for some $c_\pi=2^{O(\card{\pi})}$, where a~\emph{\mbox{$k$-mixed} minor} of a~matrix $M$ is a~$k$-by-$k$ division of $M$ into consecutive blocks of rows and columns such that each of the $k^2$ cells formed by the division has at least two distinct row vectors and at least two distinct column vectors.
We then say that such cells are \emph{mixed}.

It was cleverly observed that \emph{$k$-almost mixed minors} is a~better notion to conduct induction on~\cite{tww-quasi-chi-bounded}, where one waives the mixedness condition off for the diagonal cells. 
Besides, one can remark that mixed minors and almost mixed minors are linearly tied, in the sense that if a~matrix has a~$k$-almost mixed minor it has a~$\lfloor k/2 \rfloor$-mixed minor (while a~mixed minor is a~fortiori an almost mixed minor).

If we summarize the situation after these opening moves, we have a~permutation $\sigma$ (in fact, two permutations whose product equals $\sigma$) whose matrix encoding has no $k$-almost mixed minor with $k=2c_\pi$.
We wish to decompose this permutation as a~product of $g(k)$ permutations 
 all of which have matrix encodings without $(k-1)$-almost mixed minor.
 If this scheme works while $k \ge 3$, we will have written our permutation as a~product of $f(k)$ permutations without 2-almost mixed minor, which can be easily shown to be separable permutations.
 (Indeed the two excluded permutations 2413 and 3142 both admit a~2-mixed minor.)
\cref{thm:main} will then come from an appropriate choice for the function~$g$, which happens to make~$f$ single exponential.

\paragraph*{Delayed substitutions.}
We decompose $\sigma$ onto a~tree $T$ called \emph{delayed structured tree}.
This tree is inspired by the work of Pilipczuk and Sokołowski~\cite{tww-quasi-chi-bounded}, and was formally introduced by Bourneuf and Thomassé~\cite{tww-poly-chi-bounded} on their way to show that every graph class of bounded twin-width is polynomially $\chi$-bounded (i.e. can be properly colored with a~number of colors polynomial in their maximum clique size).
We now detail how to build an appropriate delayed structured tree in the particular case of permutations.
This will serve as an informal definition of \emph{delayed structured trees} and \emph{delayed substitutions}, crucial elements of our proof.

Here it helps to renormalize $(<,\prec)$ such that $\prec$ is the natural order over $[n]$.
This totally fixes the successor relation of $<$ to an ordering eventually matching the left-to-right order of the leaves of $T$ (as represented in~\cref{fig:delayed-dec}); leaves, which are in one-to-one correspondence with $[n]$.
Let $a_1 < a_2 < \ldots < a_n$ be the ordering of $<$.
We recursively partition $\langle a_1,a_2,\ldots,a_n \rangle$ in the following way.
At the start, the root of $T$ (the only node thus far) is labeled with the list $\langle a_1,a_2,\ldots,a_n \rangle$.

If $\{a_i,a_{i+1},\ldots,a_j\}$ is a~non-singleton interval along $\prec$ (i.e. if $\langle a_i,a_{i+1},\ldots,a_j \rangle$ can be reordered as $j-i+1$ consecutive integers), as it initially occurs with $i=1$ and $j=n$ (when $n>1$), we arbitrary split $\langle a_i,a_{i+1},\ldots,a_j \rangle$ into two non-empty lists $\langle a_i,a_{i+1},\ldots,a_k \rangle$ and $\langle a_{k+1},\ldots,a_j \rangle$, each labeling a~child of the node labeled by $\langle a_i,a_{i+1},\ldots,a_j \rangle$.
Lists $\langle a_i \rangle$ of size 1 are given a~unique child, a~leaf labeled by $a_i$.
Finally, the interesting case is when $\{a_i,a_{i+1},\ldots,a_j\}$ does not consist of a~single interval of $\prec$, but several, say, $I_1, \ldots, I_s$ with $s \ge 2$.
Then we cut $\langle a_i,a_{i+1},\ldots,a_j \rangle$ into $q+1$ sublists $\langle a_i,a_{i+1},\ldots,a_{h_1} \rangle$, $\langle a_{h_1+1},a_{h_1+2},\ldots,a_{h_2} \rangle$, $\ldots$, $\langle a_{h_q+1},a_{h_q+2},\ldots,a_j \rangle$ that are maximal for the property of being contained in a~single $I_\ell$.

Let us have a~look, in~\cref{fig:delayed-dec}, at an example of delayed structured tree if $\sigma$ is such that the ordering of $<$ is $3,1,10,5,16,9,8,4,2,6,7,11,17,15,13,12,18,14$.
\begin{figure}[!ht]
\centering
  \begin{tikzpicture}[node/.style={fill,circle,inner sep=0.06cm},
  node2/.style={circle,inner sep=0.05cm},
 arc/.style={very thick,-{>[length=2mm, width=2mm]}}]
  \def\s{1.5} 
      \foreach \i/\j\/\l in {5/6/a, 0/5/b1,10/5/b2,
       -2/4/c1,0/4/c2,2/4/c3, 
       7/4/c4,9/4/c5,11/4/c6,13/4/c7,
       -3/3/d1,-2.33/3/d2,-1.66/3/d3,-1/3/d4, 0/3/d5, 0.66/3/d6,1.33/3/d7,2/3/d8,2.66/3/d9,3.33/3/d10,
       7/3/d11,8.33/3/d12,9.66/3/d13,11/3/d14,13/3/d15,
       -3/2/e1,-2.33/2/e2,-1.66/2/e3,-1/2/e4,   0.4/2/e5,0.92/2/e6,1.33/2/e7,2/2/e8,2.4/2/e9,2.92/2/e10,3.33/2/e11,
       8.33/2/e12,9.33/2/e13,10/2/e14,
       0.4/1/f1,0.92/1/f2,2.4/1/f3,2.92/1/f4, 9.33/1/f5,10/1/f6}{
       \node[node] (\l) at (\i,\s * \j) {} ;
      }

      \foreach \i/\j in {a/b1,a/b2, b1/c1,b1/c2,b1/c3, b2/c4,b2/c5,b2/c6,b2/c7,
      c1/d1,c1/d2,c1/d3,c1/d4, c2/d5, c3/d6,c3/d7,c3/d8,c3/d9,c3/d10,
      c4/d11, c5/d12,c5/d13, c6/d14, c7/d15,
      d1/e1,d2/e2,d3/e3,d4/e4,
      d6/e5,d6/e6, d7/e7, d8/e8, d9/e9,d9/e10, d10/e11,
      d12/e12,d13/e13,d13/e14,
      e5/f1,e6/f2,e9/f3,e10/f4, e13/f5, e14/f6}{
       \draw[very thin] (\i) -- (\j) ;
      }
      
      \def\o{0.25}
      \foreach \i/\j\/\l in {-3/1.5/3,-2.33/1.5/1,-1.66/1.5/10,-1/1.5/5,
        0/2.5/16,
       0.4/0.5/9,0.92/0.5/8, 1.33/1.5/4,2/1.5/2, 2.4/0.5/6,2.92/0.5/7, 3.33/1.5/11,
       7/2.5/17, 8.33/1.5/15, 9.33/0.5/13,10/0.5/12, 11/2.5/18,13/2.5/14}{
        \node at (\i,\s * \j + \s * \o) {$\l$} ;
      }

      \def\gr{green!50!black}
      \def\lgr{green!80!black}
      \def\yellow{yellow!80!black}
      \foreach \i/\j/\b/\c in {c1/c3/15/blue,c3/c7/15/blue,c7/c5/15/blue,c5/c2/15/blue,c2/c4/15/blue,c4/c6/15/blue,
     e2/e1/10/orange,e1/e4/25/orange,e4/e3/10/orange,
     d2/d8/20/\gr,d8/d1/24/\gr,d1/d7/20/\gr,d7/d4/15/\gr,d4/d9/20/\gr,d9/d3/28/\gr,d3/d10/20/\gr,d10/d5/18/\gr,
     d13/d15/25/gray,d15/d12/25/gray,d12/d11/20/gray,d11/d14/25/gray,
     e8/e7/10/purple,e7/e9/35/purple,e9/e10/10/purple,e10/e6/20/purple,e6/e5/10/purple,e5/e11/20/purple,
     f2/f1/10/cyan,f3/f4/10/\lgr,
     e14/e13/10/\yellow,e13/e12/10/\yellow,
     f6/f5/10/magenta}{
        \draw[arc,\c] (\i) to [bend left=\b] (\j) ;
      }
      \foreach \i/\c in {a/blue,c1/orange,b1/\gr,b2/gray,c3/purple,d6/cyan,d9/\lgr,c5/\yellow,d13/magenta}{
      \node[node2,fill=\c] at (\i) {} ;   
      }
  \end{tikzpicture}
\caption{A well-chosen delayed structured tree of a~permutation $\sigma$ encoded as two linear orders $<,\prec$ over $[18]$, where $<$ is the left-to-right order on the leaves $3, 1, 10, 5, 16, 9, \ldots$ whereas $\prec$ is the natural order $1, 2, 3, \ldots$ 
The permutations $(<_t,\prec_t)$ over the grandchildren of a~node $t$ in a~given colour are drawn in the same colour by the successor relation of $\prec_t$ ($<_t$ being the left-to-right order). If all these permutations are in $\Cc$, $\sigma$ is said to be \emph{obtained by delayed substitution from~$\Cc$}.}
\label{fig:delayed-dec}
\end{figure}
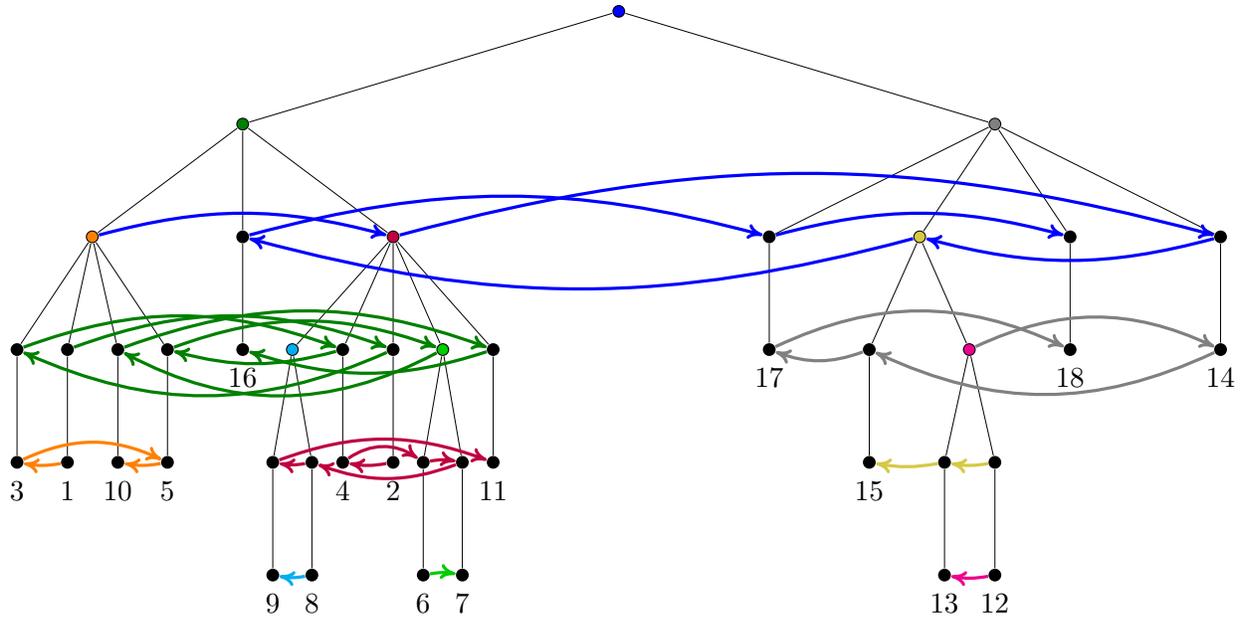
At the root, the rules specify that the list should be split in two (arbitrarily), and we accordingly chose to cut the first sublist after 11.
Moving on to the left child of the root, the current set is $[1,11] \cup \{16\}$ (two intervals), hence its three children labeled $\langle 3, 1, 10, 5 \rangle$, $\langle 16 \rangle$, $\langle 9,8,4,2,6,7,11 \rangle$.

Note that the figure also displays ``local" linear orders (in different colours) over the grandchildren of a~given node $t$ (of that same colour).
These orders $\prec_t$ are naturally inherited from $\prec$, here (among other possible options) by giving to a~grandchild the label of its leftmost descendant.
Let us focus on the blue order. 
Its successor relation defines the following ordering: 3 (leftmost descendant of the orange node), 9 (same with the purple node), 14, 15 (same with the yellow node), 16, 17, 18.
As the left-to-right order $<$ on the leaves of $T$ also naturally defines the left-to-right order $<_t$ on the grandchildren of $t$, every node $t$ (with grandchildren) naturally defines the permutation $(<_t,\prec_t)$.

Finally a~\emph{delayed substitution from a~class $\Cc$ of permutations} is such a~decomposition when every $(<_t,\prec_t)$ belongs to $\Cc$.
We show (see~\cref{lem:delayed-decomp}) that any permutation obtained by delayed substitution from $\Cc$ is the product of three permutations in $\Cc$,
as long as the class~$\Cc$ is closed under some simple operations: substitution, taking patterns, and inverse. 
We then zoom in on the permutations $(<_t,\prec_t)$ and decompose them further.

\paragraph*{Decomposing permutations of height two.}
What we gain here is that the permutations $(<_t,\prec_t)$ have a~structured tree of height~two.
Interestingly the children of $t$ naturally define a~partition $\Pc$ over the universe of $(<_t,\prec_t)$, that is, over the grandchildren of $t$.
Each part of $\Pc$ is by construction an interval along $<_t$.
We look at how these parts behave and interlace w.r.t.~$\prec_t$.

Two parts $P_1,P_2 \in \Pc$ are said \emph{mixed} if there are $x_1, y_1 \in P_1$ and $x_2, y_2 \in P_2$ such that $x_1 \prec_t x_2 \prec_t y_1 \prec_t y_2$ or $x_2 \prec_t x_1 \prec_t y_2 \prec_t y_1$.
This definition exactly matches the mixed cells in the definition of (almost) mixed minors.
This brings us to a~common theme in this context: colouring the conflict graph (red graph, or graph of the \emph{mixed} relation) with a~bounded number of colours to separately deal with large chunks without conflicting pairs.
This is usually performed directly on a~partition from a~sequence witnessing low twin-width (see for instance \cite{twin-width3,tww-approx}), where the conflict graph of the partition has bounded maximum degree.
In our case, like in~\cite{tww-poly-chi-bounded}, the subtlety is that $\Pc$ does not come from a~partition sequence, and its \emph{mixed} graph may have large maximum degree.

\paragraph*{Getting rid of the mixedness.}
Nevertheless, the mixed graph has no $K_k$ (clique on $k$ vertices) nor $K_{k,k}$ (biclique, with $k$ vertices being fully adjacent to $k$ other vertices) subgraphs, as they would entail the presence of a~$k$-almost mixed minor. 
We can split the edge set of the \emph{mixed} graph into a~bounded-degeneracy graph and an \emph{overlap graph} (also known as \emph{circle graph}); the latter is degenerate as a~consequence of excluding large clique and biclique subgraphs. 
Therefore the \emph{mixed} graph is itself $h(k)$-degenerate, for some polynomial function $h$, hence properly colourable with $h(k)+1$ colours; see~\cref{lem:mixed-degenerate}.

Once again we can zoom in on a~more structured subset of our permutation, where no pairs of parts are mixed.
This is because the so-called \emph{shuffles}---permutation operations that can recreate $(<_t,\prec_t)$ from the $h(k)+1$ ``non-mixed'' permutations---can be expressed as bounded products of separable permutations; see the more precise formulation of~\cref{lem:shuffle-decomp}.
We now want to decompose a~permutation $(<',\prec')$ in the same setting as $(<_t,\prec_t)$ but with the additional assumption that its partition $\Pc$ has no pair of mixed parts.

\paragraph*{Endgame---the induction trick on the almost mixed minor.}
Our last goal is to show that for each $P \in \Pc$ the permutation restricted to the elements of~$P$ (at the level of the grandchildren of~$t$) as well as the permutation over the parts of~$\Pc$ (at the level of the children of~$t$) are simpler than the initial permutation $\sigma$.
As announced, this takes the form of proving that these permutations have no $(k-1)$-almost mixed minor.
Once this is established, we are finally done by invoking~\cref{lem:non-mixed-decomp}, which decomposes $(<',\prec')$ into a~bounded product involving the above-mentioned permutations and separable permutations.

\begin{figure}[h!]
\centering
\begin{tikzpicture}[node/.style={fill,circle,inner sep=0.06cm},
  node2/.style={circle,inner sep=0.05cm},
 arc/.style={very thick,-{>[length=2mm, width=2mm]}}]

\def\hs{0.75} 
 \def\s{16}

\node[node] (t) at (\s * \hs / 2, 2) {} ;
\node at (\s * \hs / 2, 2.4) {$v$} ;

 \foreach \i in {1,...,\s}{
 \node[node] (v\i) at (\i * \hs,0) {} ;
 \node[node] (x\i) at (\i * \hs,-2) {} ;
 \draw (x\i) -- (v\i) -- (t) ;
 }
 \foreach \i/\j in {1/{x_1},2/{x'_1}, 5/{x_2},6/{x'_2}, 8/{x_3},9/{x'_3}, 12/{x_4},13/{x'_4}, 15/{x_5},16/{x'_5}}{
   \node at (\i * \hs, -2.4) {$\j$} ;
}
\node at (\hs-0.25,-0.25) {$v_1$} ;
\node at (2 * \hs-0.25,-0.25) {$v'_1$} ;

\begin{scope}[yshift=-0.3cm]
\def\s{0.25}
\def\h{0.1}
\foreach \p/\i/\j/\c in {0/-4.23/-0.5/red,1/1/4/black,2/5/7/black,3/8/11/black,4/12/14/black,5/15/16/black}{
   \draw[thick,\c] (\i * \hs -\s,-2.5) -- (\j * \hs +\s,-2.5) ;
   \draw[thick,\c] (\i * \hs -\s,-2.5-\h) -- (\i * \hs -\s,-2.5+\h) ;
   \draw[thick,\c] (\j * \hs +\s,-2.5-\h) -- (\j * \hs +\s,-2.5+\h) ;
   \node at (\i * \hs * 0.5 + \j * \hs * 0.5, -2.8) {\textcolor{\c}{$R_\p$}} ;
}
\end{scope}

\foreach \i/\j in {1/-3, 2/-1.2, 3/-2.1, 4/-0.23, 5/-3.6}{
\node[node] (y\i) at (\j * \hs,-2) {} ;
\node at (\j * \hs,-2.4) {$y_\i$} ;
}
\node[node] (m) at (-4.5 * \hs,-2) {} ;
\node at (-4.5 * \hs,-2.4) {$\first$} ;

\foreach \i/\j\b in {x1/y1/-20,y1/x2/45, x6/y2/-40,y2/x5/15}{
\draw[arc,gray] (\i) to [bend left=\b] (\j) ;
}
\end{tikzpicture}
\caption{Illustration of the induction trick. 
By construction of the children of $v$, $v_1$ and $v'_1$ have each a~descendant (or simply child) $x_1$ and $x'_1$, respectively, such that there is a~$y_1$ outside the descendants of $v$ in between $x_1, x'_1$ in the order $\prec'$.
Assume $R_1, \ldots, R_5$ is a~5-almost mixed minor on a~choice of one descendant per child of $v$.
For the sake of clarity, let us make the simplifying assumptions that the minor is symmetric ($C_i=R_i$ for every $i$), every $y_i$ is to the left of $v$, and thus can be grouped together with $\first$, the minimum along $<'$, in an interval $R_0=C_0$ of $<'$, and $\first \prec' x$ for every $x$ descendant of~$v$.
It then holds that $R_0$ is mixed with each $R_i$ (for $i \in [5]$) since for instance $\first \prec' x_1 \prec' y_1 \prec' x'_1$ and $\first \prec' x'_2 \prec' y_2 \prec' x_2$.
}
\label{fig:induction-trick}
\end{figure}
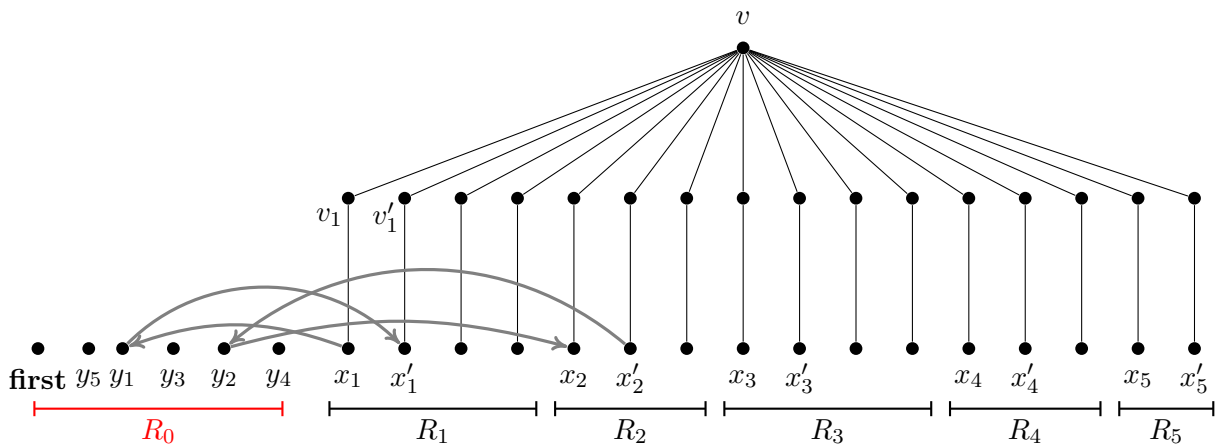

We are left with the task of building a~$k$-almost mixed minor of $\sigma$ (a~contradiction) from a~supposed $(k-1)$-almost mixed minor $\Dc$ over the children of some node $v$ of $T$.
Let us recall that ``starting a~new sibling'' in the construction of $T$ means that the last element of the current child and the first element of the new child are in two distinct intervals $I_a, I_b$.
Hence there is an element $y$ outside the children of $v$ such that $I_a \prec' y \prec' I_b$.
We get one such $y$ for every row part and every column part of $\Dc$.
By leveraging again the \emph{shuffle} trick, we can reduce to the case when all such vertices are to the left (by $<'$) of the children of $v$ (see~\cref{thm:main-induction}).
Then the minimum of~$<'$ together with these vertices $y$ can make a~$k$-th interval (in rows and columns) that is mixed with every previous part of $\Dc$, thereby yielding a~$k$-almost mixed minor (see~\cref{fig:induction-trick}).

\subsection{Perspectives}
We prove that for permutations, the minimum size of an avoided pattern,
and the minimum length of a~factorisation into separable permutations, are functionally equivalent.
However, the bounds we give are huge, the dependency being exponential in one direction, and doubly exponential in the other,
leaving open the question of finding tighter bounds.
While avoided patterns have been extensively studied,
there is to our best knowledge no work on the length of factorisations into separable permutations.
A natural question is whether this new parameter can be computed efficiently.

In a~different direction, a~natural generalisation of our result would be to extend this factorisation to matrices.
It was shown in~\cite{twin-width4} that a~class~$\cal M$ of 0,1-matrices closed under taking submatrices
has growth at~most single-exponential if and only if~$\cal M$ has bounded twin-width.
It is then natural to ask if every matrix of~$\cal M$ can be factorised into
a~bounded product (say, computed in $\mathbb F_2$) of some basic 0,1-matrices playing the role of separable permutations.


Finally, a~crucial point of this work is perhaps the used strategy:
To decompose a~pattern-free permutation $(<_1,<_2)$, first find a~third order $<_3$ compatible with both orders (i.e. such that $(<_1,<_3)$ and $(<_2,<_3)$ are $k$-mixed minor free),
and then express both $(<_1,<_3)$ and $(<_2,<_3)$ as delayed structured trees (see \cref{fig:delayed-dec}).
How delayed substitutions untangle the complexity of permutations and other binary structures
is a~novel and promising tool that, we believe, can be exploited in many other ways.

\subsection{Organisation of the paper}

In \cref{sec:prelim}, we introduce a~number of definitions, conventions, and simple results, which we use throughout this work.
Notably we recall the definitions of ordered trees, substitutions and shuffles of permutations, degeneracy and proper colourings in the context of sparse circle graphs, and relevant properties of classes with bounded twin-width.
Note that we will \emph{not} actually need the definition of twin-width and of transductions (they can both be found for instance in~\cite{twin-width1}).

\Cref{sec:delayed} focuses on delayed substitutions, detailing how to construct them, and how to decompose them into products and substitutions.
\Cref{sec:quotients} shows how to further decompose the permutations that appear in the delayed substitution, by studying their mixed graphs.
The core of the proof of \cref{thm:main} is in \cref{sec:decomposition},
using lemmas from the previous sections to decompose $k$-almost mixed free permutations into $(k-1)$-almost mixed free ones.
In \cref{sec:algo}, we explain how to implement \cref{thm:main} in linear time.
Finally, \cref{sec:graphs} shows how to apply \cref{thm:main} to graphs,
to prove \cref{thm:sparsetww-subdivision,thm:tww-FO-desc}.

\section{Preliminaries}\label{sec:prelim}
This section defines our notations and conventions,
and introduces some important constructions on permutations.
We denote by~$[n]$ the interval of integers~$\{1,\dots,n\}$.

\subsection{Ordered trees}

Trees are used throughout this work to describe decompositions of permutations.
We work with rooted trees. Vertices of the trees are \emph{nodes}.
The \emph{ancestors} of a~node $x$ are the nodes in the unique path from $x$ to the root $r$,
and the \emph{parent} of $x$ is the first node on this path.
We will also speak of \emph{descendants}, \emph{children},
\emph{grandchildren}, \emph{siblings} (nodes with same parent),
and \emph{cousins} (non-sibling nodes with the same grandparent).
The set of leaves of a~tree~$T$ is denoted by~$L(T)$.
For any node~$t \in T$, we denote by~$T(t)$ the subtree rooted at~$t$ (i.e. consisting of the descendants of~$t$),
and by~$L(t) \subseteq L(T)$ the set of leaves of~$T(t)$.

Here an \emph{ordered tree}~$(T,<)$ is a~rooted tree~$T$ equipped with a~linear order~$<$ on $L(T)$,
such that for each node~$t \in T$, the leaves~$L(t)$ form an interval of~$<$.
In that case we also say that~$<$ is \emph{compatible} with~$T$.
It is natural to think of~$<$ as a~\emph{left-to-right} order.

If~$(T,<)$ is an ordered tree and $x,y \in T$ are not in an ancestor--descendant relationship,
then $L(x),L(y)$ are disjoint, and each of them is an interval for~$<$,
hence either $L(x) < L(y)$, or $L(y) < L(x)$.
We then naturally extend~$<$ to~$x,y$ by~$x < y$ in the former case, and~$y < x$ in the latter.
In particular, $<$ induces a~linear order~$<_t$ on the children of any given internal node~$t$.
This can be reversed.
Given a~tree~$T$, and an order~$<_t$ on the children of~$t$ for each internal node~$t$,
one can define~$<$ compatible with~$T$ as follows:
for any leaves~$x,y$, we have $x < y$ iff $x' <_t y'$, where $t$ is the last (that is, closest) common ancestor of~$x,y$,
and~$x'$ and $y'$ are the children of~$t$ that are ancestors of~$x$ and $y$, respectively.

\subsection{Permutations and biorders}

Throughout the paper, we will represent permutations as \emph{biorders},
i.e.\ the superposition of two linear orders on the same set.
Precisely, a~permutation $\sigma \in \perm_n$ is represented by the biorder $B_\sigma = ([n], <, <_\sigma)$,
where~$<$ is the natural order on~$[n]$, and $x <_\sigma y$ iff $\sigma(x) < \sigma(y)$.
Conversely, from a~biorder $B = (X,<_1,<_2)$ on~$n$ elements,
we define the associated permutation $\sigma_B$ such that
if $x\in X$ is in $i$th position for $<_1$, then it is in position $\sigma_B(i)$ for $<_2$.
This is a~bijection between permutations, and biorders considered up to isomorphism.
In the proofs, this bijection will be left implicit, and we may for example say
that a~biorder belongs to some class of permutations.

The inverse of a~permutation is obtained by swapping the roles of the two orders:
\begin{equation}
    \sigma_{(X,<_1,<_2)}^{-1} = \sigma_{(X,<_2,<_1)}.
\end{equation}
Similarly, factoring~$(X,<_1,<_2)$ corresponds to choosing a~third order~$<_3$:
\begin{equation}
    \sigma_{(X,<_3,<_2)} \circ \sigma_{(X,<_1,<_3)} = \sigma_{(X,<_1,<_2)}.
\end{equation}

The notion of pattern has a~very natural definition in terms of biorders.
A~pattern of $(X,<_1,<_2)$ is any induced restriction $(Y,<_1,<_2)$ for some subset $Y \subseteq X$.

\subsubsection{Substitutions}

Let us now introduce \emph{substitutions} of permutations, which will be used throughout this work,
and are a~natural introduction to the more general delayed decompositions presented in \cref{sec:delayed}.
Substitutions are easily visualised through permutation matrices:
the substitution of~$\tau$ inside~$\sigma$ corresponds to
replacing one of the 1-entries in the matrix of~$\sigma$ by the matrix of~$\tau$.
In this work, we are interested in \emph{iterated substitutions}---%
i.e.\ starting from a~permutation~$\sigma$ in a~class~$\Cc$,
and iteratively substituting other permutations of~$\Cc$ inside it---%
and we will consistently use substitution to mean iterated substitutions.

A precise description of an (iterated) substitution can be given by a~tree as follows.
Let~$T$ be a~tree, and~$<,\prec$ two linear orders on~$L(T)$, both compatible with~$T$.
For every internal node~$t \in T$, they induce orders~$<_t,\prec_t$ on the children of~$t$.
Then the permutation~$(<,\prec)$ is obtained by substitution of all the~$\{(<_t,\prec_t)\}_{t \in T \setminus L(T)}$.
In this case, we also say that~$(<,\prec)$ is obtained by \emph{substitution along~$T$}.
Conversely, any iterated substitution process can be expressed through such a~tree.

For a~class~$\Cc$ of permutations, let~$\Cc^*$ denote the \emph{substitution closure} of~$\Cc$, that is,
$\sigma$ is in~$\Cc^*$ if it is obtained by substitution of permutations in~$\Cc$.
We say that~$\Cc$ is \emph{closed under substitution} if~$\Cc = \Cc^*$.
Note that~$\left(\Cc^*\right)^* = \Cc^*$, meaning that the substitution closure is indeed closed under substitution.
For example, the class~$\Sep$ of \emph{separable} permutations
is the substitution closure of the class~$\{12,21\}$ of permutations on two elements.
It happens that~$\Sep$ is precisely the class of permutations excluding the patterns~$3142$ and~$2413$~\cite{BoseBL98}.
This implies that the inverse of a~separable permutation is also a~separable permutation.

\begin{lemma}
    \label{lem:product-substitution}
    If~$\Cc,\Dc$ are substitution-closed classes of permutations,
    then so are~$\Cc \circ \Dc$ and~$\Cc^{-1}$.
\end{lemma}
\begin{proof}
    Let us prove more generally that for any classes~$\Cc,\Dc$, the following holds:
    \begin{equation}
        \label{eq:subst-product}
        \left(\Cc^{-1}\right)^* = \left(\Cc^*\right)^{-1} 
        \qquad \text{and} \qquad
        (\Cc \circ \Dc)^* \subseteq \Cc^* \circ \Dc^*.
    \end{equation}
    Let~$\sigma \in \left(\Cc^{-1}\right)^*$.
    Then,~$\sigma$ is obtained by substitution of~$\{(<_t, \prec_t)\}_{t \in T \setminus L(T)}$ along a~tree~$T$, 
    and for every internal node $t$ the permutation~$(<_t, \prec_t)$ is in~$\Cc^{-1}$.
    Let~$<, \prec$ be the corresponding linear orders on~$L(T)$.
    Then,~$\sigma = (<, \prec)$, hence~$\sigma^{-1} = (\prec, <)$.
    However,~$(\prec, <)$ is obtained by substitution of $\{(\prec_t, <_t)\}_{t \in T \setminus L(T)}$ along~$T$, 
    and for every internal node~$t$, $(\prec_t, <_t) = (<_t, \prec_t)^{-1}$,
    so~$(\prec_t, <_t) \in \Cc$.
    Thus $\sigma^{-1} \in \Cc^*$, and~$\sigma \in (\Cc^*)^{-1}$.
    This proves that~$\left(\Cc^{-1}\right)^* \subseteq \left(\Cc^*\right)^{-1}$.
    Applying this result to~$\Cc^{-1}$, we get~$\Cc^* \subseteq \left(\left(\Cc^{-1}\right)^*\right)^{-1}$.
    This immediately yields~$\left(\Cc^*\right)^{-1} \subseteq \left(\Cc^{-1}\right)^*$.

    Let~$\sigma \in (\Cc \circ \Dc)^*$.
    Then,~$\sigma$ is obtained by substitution of~$\{(<_t, \prec_t)\}_{t \in T \setminus L(T)}$ along a~tree~$T$, 
    and for every internal node $t$ the permutation~$(<_t, \prec_t)$ is in~$\Cc \circ \Dc$.
    Thus, for every such~$t$, there exists a~linear order~$<'_t$ on the children of~$t$
    such that~$(<_t, <'_t) \in \Dc$ and~$(<'_t, \prec_t) \in \Cc$.
    Let~$<, <'$ and~$\prec$ be the corresponding linear orders on~$L(T)$.
    Then,~$\sigma = (<, \prec) = (<', \prec) \circ (<, <')$.
    However,~$(<', \prec)$ is obtained by substitution of
    $\{(<'_t, \prec_t)\}_{t \in T \setminus L(T)}$ along~$T$,
    so~$(<', \prec) \in \Cc^*$.
    Similarly, $(<, <') \in \Dc^*$, thus $\sigma \in \Cc^* \circ \Dc^*$.
    
    When~$\Cc,\Dc$ are closed under substitution,
    the inclusions \eqref{eq:subst-product} simplify to $\left(\Cc^{-1}\right)^* = \Cc^{-1}$
    and $(\Cc \circ \Dc)^* \subseteq \Cc \circ \Dc$, proving the result.
\end{proof}

\subsubsection{Shuffles}
A~permutation~$\sigma \in \perm_n$ is a~\emph{$k$-shuffle}
if there is some partition $\biguplus_{i=1}^k X_i$ of its domain $[n]$ such that
the restriction (in the sense of patterns) of~$\sigma$ to any~$X_i$ is the identity permutation.
More generally, given a~class~$\Cc$ of permutations,
$\sigma$ is a~\emph{$k$-shuffle of~$\Cc$} if there is a~partition $\biguplus_{i=1}^k X_i = [n]$
such that the restriction of~$\sigma$ to any~$X_i$ is in~$\Cc$.

Let us say that a~class~$\Cc$ of permutations is \emph{closed under symmetry}
if it is closed by conjugating with the decreasing permutation:
if~$\sigma \in \Cc$, then the permutation~$i \mapsto n+1 - \sigma(n+1 - i)$ should also be in~$\Cc$.
In terms of biorders, this corresponds to replacing each of the two linear orders by its dual, i.e.\ its reverse order.
Remark that if~$\Cc, \Dc$ are closed under symmetry, then so is~$\Cc \circ \Dc$.
For example, the class $\Sep$ of separable permutations is closed under symmetry.
\begin{lemma}
    \label{lem:shuffle-decomp}
    For a~class~$\Cc$ of permutations closed under substitution and symmetry,
    any $2^k$-shuffle of~$\Cc$ is in $\Sep^k \circ \Cc \circ \Sep^k$.
    In particular, any $2^k$-shuffle is in~$\Sep^{2k}$.
\end{lemma}
\begin{proof}
    Let~$\sigma = (X,<_1,<_2)$ be a~$2^k$-shuffle of~$\Cc$.
    Then there is a~partition of its domain $X = Y \uplus Z$ such that
    the restrictions of~$\sigma$ to both~$Y$ and~$Z$ are~$2^{k-1}$-shuffles of~$\Cc$.
    
    Define an intermediate linear order~$\prec_1$ by
    \begin{compactitem}
        \item $Y \prec_1 Z$,
        \item inside~$Y$, $\prec_1$ coincides with~$<_1$, and
        \item inside~$Z$, $\prec_1$ coincides with the dual of~$<_1$.
    \end{compactitem}
    We claim that~$(X,<_1,\prec_1)$, is separable.
    Indeed, let~$x$ be the minimum of~$X$ for~$<_1$.
    If \mbox{$x \in Y$}, then~$x$ is the minimum of~$\prec_1$,
    while if~$x \in Z$ it is the maximum of~$\prec_1$.
    Either way, we can separate~$X$ into~$\{x\}$ and~$X \setminus \{x\}$,
    and proceed by induction on the latter.
    This scheme is a~decomposition of~$(X,<_1,\prec_1)$ as a~separable permutation,
    with the specificity that the separating tree is a~caterpillar.
    We define~$\prec_2$ similarly with regards to~$<_2$, so that~$(X,\prec_2,<_2)$ is separable.

    Consider now the permutation~$(X,\prec_1,\prec_2)$.
    Since we have~$Y \prec_i Z$ for both~$i=1,2$,
    it suffices to consider this permutation restricted to either~$Y$ or~$Z$.
    By construction, the restriction to~$Y$ is exactly~$(Y,<_1,<_2)$,
    and the restriction to~$Z$ is~$(Z,<_1,<_2)$ up to symmetry.
    These two permutations are $2^{k-1}$-shuffles of~$\Cc$,
    hence by induction are in $\Sep^{k-1} \circ \Cc \circ \Sep^{k-1}$.
    Since this class is closed under symmetry and substitution,
    we obtain that~$(X,\prec_1,\prec_2)$ is also in $\Sep^{k-1} \circ \Cc \circ \Sep^{k-1}$,
    and conclude by composing with~$(X,<_1,\prec_1)$ and~$(X,\prec_2,<_2)$.

    Finally, the remark that $2^k$-shuffles are in~$\Sep^{2k}$
    is obtained by applying the result with~$\Cc$ being the class of identity permutations.
\end{proof}

\subsection{Twin-width}\label{sec:twinwidth}
Guillemot and Marx~\cite{guillemot14patterns} introduced a~width of permutations, that was later generalised to graphs, matrices, and binary structures by Bonnet et al.\ in~\cite{twin-width1} under the name of twin-width.

The twin-width of a~graph~$G = (V,E)$ is defined through \emph{partition sequences} (equivalent to so-called \emph{contraction sequences}),
that is, sequences~$\Pc_n,\dots,\Pc_1$ of partitions of~$V$ such that
\begin{compactenum}
\item $\Pc_n = \{\{x\}~:~x \in V\}$ is the partition into singletons,
\item $\Pc_1 = \{V\}$ is the trivial partition, and 
\item $\Pc_i$ is obtained from~$\Pc_{i+1}$ by merging two parts.
\end{compactenum}
Two parts~$X \neq Y$ in~$\Pc_i$ are \emph{homogeneous} if there are
either no edges or all possible edges between~$X$ and~$Y$ in~$G$.
If~$X,Y$ are not homogeneous (i.e.\ there is at least one edge and one non-edge between them),
they are said to be \emph{in error}.
The \emph{width} of the partition sequence~$\Pc_n,\dots,\Pc_1$
is the maximum over~$i \in [n]$ and~$X \in \Pc_i$ of the number of parts of~$\Pc_i$ in error with~$X$%
---called the \emph{error degree of~$X$}.
Finally, the twin-width of~$G$ is the minimum width of a~partition sequence.
This definition readily generalises to \emph{binary relational structures}~$(V,R_1,\dots,R_k)$,
consisting of a~domain, or vertex set~$V$, and a~number of binary relations~$R_i \subseteq V^2$.
Remark in particular that biorders are binary relational structures.
In this setting, parts~$X,Y \in \Pc_i$ are \emph{in error w.r.t.\ the relation~$R_j$}
if there are~$x_1,x_2 \in X$, $y_1,y_2 \in Y$
such that~$(x_1,y_1) \in R_j$ and~$(x_2,y_2) \not\in R_j$,
or symmetrically~$(y_1,x_1) \in R_j$ and~$(y_2,x_2) \not\in R_j$.
The width of the partition sequence~$\Pc_n,\dots,\Pc_1$
is now the maximum over all~$j \in [k]$, $i \in [n]$, and~$X \in \Pc_i$
of the error degree of~$X$ in~$\Pc_i$ w.r.t.~$R_j$.

As fundamental property, twin-width exhibits a~duality with grid-like structures in matrices---or universal patterns of permutations.
This result, based on the Marcus--Tardos theorem~\cite{marcustardos},
appears in the seminal work of Guillemot and Marx, and can be restated as follows.
\begin{theorem}[{\cite[Theorem~4.1]{guillemot14patterns}}]
    \label{thm:pattern-tww}
    For any permutation~$\pi$, there is~$c_\pi = 2^{O(|\pi|)}$ such that
    if~$\sigma$ avoids~$\pi$ as pattern,
    then the biorder~$B_\sigma$ associated with~$\sigma$ has twin-width at most~$c_\pi$.
\end{theorem}
The bound~$c_\pi = 2^{O(|\pi|)}$ in the restated theorem
follows from a~later improvement in the Marcus--Tardos constant due to Fox~\cite{fox2013stanleywilf}.
A \emph{division}~$\Dc$ of a~matrix consists of partitions~$\Rc,\Cc$
of its rows and columns respectively into \emph{intervals}.
It is a~\emph{$k$-division} if the partitions have~$k$ parts each.
A \emph{cell} of the division is the submatrix induced by~$X \cap Y$ for some~$X \in \Rc, Y \in \Cc$.
A $k$-\emph{grid} in a~0,1-matrix is a~$k$-division in which every cell contains a~1-entry.
Remark that the order of rows and columns in crucial in these definitions
because divisions are required to be partitions into intervals.

The notion of grid is tightly linked to permutation patterns, and to \cref{thm:pattern-tww}.
In order to generalise the latter to dense matrices, Bonnet et al.~\cite{twin-width1} introduced \emph{mixed minors}.
Say that a~matrix is \emph{horizontal} if all rows are constant,
i.e.\ in each row, all the coefficients are the same (or equivalently, all columns of the matrix are equal).
Symmetrically, it is \emph{vertical} if all columns are constant,
and finally it is \emph{mixed} if it is neither horizontal nor vertical.
Then, a~$k$-\emph{mixed minor} in a~matrix is a~$k$-division in which every cell is mixed.
A matrix is said \emph{$k$-mixed free} if it does not have a~$k$-mixed minor.
The generalisation of \cref{thm:pattern-tww} can then be stated as follows~\cite[Theorem~5.8]{twin-width1}:
There is a~function $f$ such that if a~graph~$G$ admits a~$k$-mixed free adjacency matrix, then~$G$ has twin-width at most~$f(k)$.

We will not need this result, but only its much simpler-to-prove converse.
\begin{lemma}[{\cite[first part of Theorem~5.4]{twin-width1}}]
  \label{lem:tww-mixed-minor}
  If~$(V,R_1,\dots,R_k)$ is a~binary relational structure with twin-width~$t$,
  then there exists a~linear order~$<$ on~$V$ such that
  the adjacency matrix of each relation~$R_i$, ordered by~$<$, is $(2t+2)$-mixed free.
\end{lemma}

Pilipczuk and Sokołowski introduced the following subtle relaxation of mixed minors in~\cite{tww-quasi-chi-bounded}.
Consider a~$k$-division $\Dc = (\Rc,\Cc)$ of a~matrix~$M$,
and enumerate the parts as $\Rc = \{R_1,\dots,R_k\}$, $\Cc = \{C_1,\dots,C_k\}$,
following the order of rows and columns.
We say that~$\Dc$ is a~$k$-\emph{almost mixed minor} if every cell of~$\Dc$ is mixed,
except possibly the diagonal cells~$R_i \cap C_i$ for $i \in [k]$.
Clearly a~$k$-mixed minor is also a~$k$-almost mixed minor,
and it is not hard to check that a~$2k$-almost mixed minor yields a~$k$-mixed minor by merging the parts by pairs.
Thus these two notions are equivalent up to a~factor of two.
The point of almost mixed minors is to allow a~proof by induction:
Ideally, a~structure without $k$-almost mixed minors, i.e. \emph{$k$-almost mixed free}, is decomposed
into smaller parts that have no $(k-1)$-almost mixed minors.
The same scheme is much harder to achieve using (regular) mixed minors.

\subsection{Graphs, colouring, degeneracy}
A $k$-colouring of a~graph~$G$ is a~map~$c : V
(G) \to [k]$ that assigns distinct colours to adjacent vertices.
Equivalently, it is a~partition into~$k$ colour classes~$c^{-1}(1),\dots,c^{-1}(k)$,
each of which is an independent set (i.e.\ a~set of pairwise non-adjacent vertices).

A graph~$G$ is \emph{$k$-degenerate} if~$G$ and all its subgraphs have a~vertex of degree at most~$k$.
A~well-known characterisation is that~$G$ is $k$-degenerate
if and only if there is an acyclic orientation of its edges such that all vertices have out-degree at most~$k$.
Degenerate graphs can be efficiently coloured in the following sense.
\begin{lemma}[\cite{MatulaB83}]
    \label{lem:degen-colouring}
    Any $k$-degenerate graph~$G$ can be $(k+1)$-coloured by a~greedy algorithm.
    Further, if the graph is given through adjacency lists,
    then this colouring can be computed in time~$O(k \card{V(G)})=O(\card{E(G)})$.
\end{lemma}

A closely related parameter is the \emph{maximum edge density} of~$G$, that is, the maximum, taken over all subgraphs~$H$ of~$G$, of~$\card{E(H)} / \card{V(H)}$.
It is easy to see that $k$-degenerate graphs have maximum edge density at most~$k$,
and conversely graphs with maximum edge density~$k$ are $2k$-degenerate.
If~$G$ admits an orientation (which may contain cycles) in which vertices have out-degree at most~$k$,
then its maximum edge density is at most~$k$.

\subsubsection{Circle graphs}
Let~$\Ic$ be a~family of intervals of some linear order.
Two intervals~$A,B$ \emph{overlap} if they intersect, but neither contains the other.
The \emph{overlap graph} of~$\Ic$ is the graph whose vertex set is~$\Ic$,
and where intervals are adjacent exactly when they overlap.
Overlap graphs are also known as circle graphs,
because they can be described by the intersections of a~set of chords of a~circle.

Let~$K_t$ denote the complete graph on~$t$ vertices,
and~$K_{t,t}$ the balanced complete bipartite graph on~$2t$ vertices.
We say that~$G$ is $H$-subgraph-free (resp.~$H$-free) to mean that~$G$ does not have a~subgraph (resp.~induced subgraph) isomorphic to~$H$.
The following lemma shows that circle graphs excluding bicliques (and thus, cliques) as subgraphs are degenerate.
Since we will use it for colouring, it is worth comparing it to a~result of Davies and McCarty stating
that $K_t$-free circle graphs have bounded chromatic number~\cite{davies2021circle,davies2022circle}.
Their hypothesis is weaker (arbitrary bicliques are allowed),
and the $O(t \log t)$ bound in~\cite{davies2022circle} is better,
but this bound applies only to the chromatic number and not to the degeneracy, and the proof is far more complex.
\begin{lemma}
    \label{lem:circle-graph-degenerate}
    Circle graphs that are $K_t$-free and $K_{t,t}$-subgraph-free
    have maximum edge density at most $2(t-1)^2$.
\end{lemma}
\begin{proof}
    Let~$\Ic$ be a~family of intervals whose overlap graph is $K_t$-free and $K_{t,t}$-subgraph-free.
    We will define an orientation of the edges with out-degree at most~$2(t-1)^2$.

    Consider two overlapping intervals~$A,B \in \Ic$.
    Relative to this edge, say that~$C \in \Ic$ is a~\emph{private ancestor of~$A$}
    if~$A \subseteq C$ and~$B \not\subseteq C$, and symmetrically with~$B$.
    We consider that~$A$ and~$B$ are also private ancestors of themselves.
    \begin{claim}
        Every private ancestor of~$A$ overlaps with every private ancestor of~$B$.
    \end{claim}
    \begin{claimproof}
        Let~$C,D$ be private ancestors of~$A,B$ respectively, with possibly~$C = A$ or~$D = B$.
        Pick~$x \in a~\cap B$ (which is non-empty as they overlap),
        and~$y \in a~\setminus D$ (which is non-empty by assumption),
        and symmetrically~$z \in B \setminus C$.
        Then, because~$A \subseteq C$ and~$B \subseteq D$,
        we also have~$x \in C \cap D$, $y \in C \setminus D$, and~$z \in D \setminus C$,
        hence~$C$ and~$D$ overlap.
    \end{claimproof}
    Thus, there is a~biclique between~$A$ and its private ancestors on the one hand,
    and~$B$ and its private ancestors on the other.
    Since the overlap graph does not contain a~$K_{t,t}$ subgraph,
    it follows that either~$A$ or~$B$ has less than~$t$ private ancestors.
    We orient the edge~$AB$ towards~$A$ if it has less than~$t$ private ancestors,
    and towards~$B$ otherwise.

    Let us bound the out-degree of this orientation.
    Fix~$I \in \Ic$, an interval with endpoints $x < y$,
    and consider~$N^+(I)$ the set of~$A \in \Ic$ with an edge oriented from~$I$ to~$A$.
    Any~$A \in N^+(I)$ contains either~$x$ or~$y$, but not both.
    Let~$X$ be the set of intervals in~$N^+(I)$ containing~$x$,
    and consider the inclusion poset~$(X,\subseteq)$.
    \begin{claim}
        The poset~$(X,\subseteq)$ does not contain a~chain of length~$t$.
    \end{claim}
    \begin{claimproof}
        If~$A_1 \subsetneq \dots \subsetneq A_s$ are in~$X$,
        then~$A_1$ overlaps with~$I$, and each~$A_i$ is a~private ancestor of~$A_1$ w.r.t.\ the edge~$A_1I$.
        From the orientation of the edge from~$I$ to~$A_1$,
        we know that~$A_1$ has less than~$t$ private ancestors, hence $s < t$.
    \end{claimproof}
    \begin{claim}
        The poset~$(X,\subseteq)$ does not contain an antichain of size~$t$.
    \end{claim}
    \begin{claimproof}
        If~$A,B \in X$ are incomparable for inclusion, then they overlap because both contain~$x$.
        Thus an antichain in the poset~$(X,\subseteq)$ is a~clique in the overlap graph.
    \end{claimproof}
    By Dilworth's theorem, it follows from the two claims that~$\card{X} \le (t-1)^2$.
    The same reasoning applies to intervals in~$N^+(I)$ containing~$y$,
    and we obtain~$\card{N^+(I)} \le 2(t-1)^2$.
\end{proof}

Finally, for algorithmic purposes, let us sketch
how to efficiently compute circle graphs from their interval representation.
\begin{lemma}
    \label{lem:circle-graph-construction}
    Given~$\Ic$ a~family of intervals in~$[n]$,
    one can compute the overlap graph~$G$ of~$\Ic$
    in time $O(n+\card{\Ic}+\card{E(G)})$.
\end{lemma}
\begin{proof}
    We iterate for~$i$ from~$1$ to~$n$,
    while maintaining the list~$L$ of intervals containing~$i$, sorted by their right endpoint.
    In the~$i$th iteration, we consider all intervals of the form~$[i,j]$ sorted by \emph{decreasing} values of~$j$, and naively insert each of them in~$L$.
    Suppose that~$I = [i,j]$ is thus inserted into~$L$ in position~$k$, after~$J_1,\dots,J_{k-1}$.
    For $\ell < k$, if~$J_\ell = [a_\ell,b_\ell]$, then it must be that $a_\ell < i \le b_\ell < j$.
    It follows that~$I$ overlaps with~$J_1,\dots,J_{k-1}$.
    Furthermore, any interval~$[a,b]$ which overlaps with~$I$ and satisfies~$a < i$ will be one of the~$J_\ell$.
    Thus, during the insertion of~$I$, we can output all edges from~$I$ to intervals ``to its left'',
    and the cost of the insertion of~$I$ is proportional to the number of such edges.
    Finally, having inserted all intervals with left endpoint~$i$,
    we remove from the beginning of~$L$ all intervals with right endpoint~$i$,
    before proceeding with~$i+1$.
\end{proof}

\subsection{First-order transductions}
For the sake of \cref{thm:tww-FO-desc}, let us introduce first-order logic, interpretations, and transductions.

A \emph{relational signature} is a~set~$\Sigma = \{R_1,\dots,R_k\}$ of \emph{relation symbols},
each with an associated \emph{arity}~$\arity(R_i) \in \Nn$.
A \emph{$\Sigma$-structure}~$S$ is defined by a~\emph{domain} (or vertex set)~$V(S)$,
and, for each relation symbol~$R \in \Sigma$ of arity~$r := \arity(R)$,
a \emph{realisation} $R(S) \subseteq V(S)^r$ of this relation.
For example, (simple undirected) graphs are structures over the signature~$\{E\}$ with~$\arity(E) = 2$
(with the specificity that the realisation of~$E$ is symmetric and irreflexive),
and biorders are structures over the signature~$\{<_1,<_2\}$, $\arity(<_i) = 2$
(the realisations being total orders).
When all relation symbols have arity~2, the signature and its structures are called \emph{binary}.
Recall from \cref{sec:twinwidth} that binary structures are the objects for which twin-width is defined.

An \emph{FO formula}~$\phi$ over the language of~$\Sigma$ can quantify on vertices,
and test whether a~relation~$R \in \Sigma$ holds for a~given tuple of vertices.
For example, over the language of graphs, the formula
\[ \phi(x,y) = E(x,y) \lor \exists z~E(x,z) \land E(z,y) \]
expresses that~$x,y$ are at distance at most~two.
If~$\phi$ is a~formula over the language of~$\Sigma$ and~$S$ is a~$\Sigma$-structure,
then~$S \models \phi$ denotes that~$\phi$ is \emph{satisfied} by~$S$, which is defined in the obvious way.

Given two signatures~$\Sigma,\Gamma$, an \emph{FO interpretation}~$\Phi$ from~$\Sigma$ to~$\Gamma$
is a~map, defined in FO logic, from $\Sigma$-structures to $\Gamma$-structures.
Precisely, $\Phi$ is described by giving
\begin{enumerate}
    \item for each relation~$R \in \Gamma$, a~formula~$\phi_R(x_1,\dots,x_r)$ over the language of~$\Sigma$,
        with as many free variables as the arity $r := \arity(R)$ of~$R$, and
    \item one last formula~$\phi_{dom}(x)$ with one free variable, again over the language of~$\Sigma$.
\end{enumerate} 
Given a~$\Sigma$-structure~$S$, its image~$\Phi(S)$ is defined as follows.
\begin{enumerate}
    \item The domain consists of vertices satisfying~$\phi_{dom}$,
    i.e.\
    \[ V(\Phi(S)) = \{x \in V(S) \ : \ S \models \phi_{dom}(x) \}. \]
    \item For each symbol~$R \in \Gamma$, the realisation~$R(\Phi(S))$ is described by~$\phi_R$,
    i.e.\ for~$x_1,\dots,x_r \in V(\Phi(S))$,
    \[ (x_1,\dots,x_r) \in R(\Phi(S)) \quad \iff \quad S \models \phi_R(x_1,\dots,x_r). \]
\end{enumerate}

\emph{Transductions} are a~non-deterministic generalisation of interpretations.
Let~$\Sigma$ be a~signature, and~$C_1,\dots,C_k$ be~$k$ unary relation symbols (i.e.\ with arity~1), disjoint from~$\Sigma$.
The $k$-colouring is the one-to-many operation which maps a~$\Sigma$-structure $S$
to all possible extensions of~$S$ as $(\Sigma \uplus \{C_1,\dots,C_k\})$-structures~$S^+$,
meaning that~$V(S) = V(S^+)$ and~$R(S) = R(S^+)$ for any~$R \in \Sigma$,
while the~$C_i(S^+)$ are chosen to be arbitrary subsets of~$V(S^+)$.
An \emph{FO transduction} is the composition of a~$k$-colouring (with~$k$ fixed), followed by an FO interpretation.

It is folklore that interpretations and transductions can be composed.
\begin{lemma}
    \label{lem:transduction-comp}
    If~$\Phi,\Psi$ are FO transductions (resp.\ interpretations) from~$\Sigma$ to~$\Gamma$ and~$\Gamma$ to~$\Delta$ respectively,
    then the composition~$\Psi \circ \Phi$ is an FO transduction (resp.\ interpretation) from~$\Sigma$ to~$\Delta$.
\end{lemma}

Finally, transductions preserve bounded twin-width in the following sense.
\begin{theorem}[{\cite[Theorem~8.1]{twin-width1}}]
    \label{thm:transduction-tww}
    For any FO transduction~$\Phi$, there is a~function~$f : \Nn \to \Nn$ such that
    for any binary structures~$S$ and~$T \in \Phi(S)$, $\tww(T) \le f(\tww(S))$.
\end{theorem}

\section{Delayed substitutions}\label{sec:delayed}
In this section, we recall the definition of delayed decompositions
from \cite[Section 2]{tww-poly-chi-bounded}, which generalise substitutions.
Since we consider biorders and not arbitrary binary relations,
we obtain some additional structure on these delayed decompositions.
We show that they can be expressed as a~bounded product of substitutions.

\subsection{Definition}
A \emph{delayed structured tree}~$\left(T,<,\{\prec_t\}_{t \in T}\right)$
consists of an ordered tree~$(T,<)$, equipped with, for each node~$t \in T$,
a linear order~$\prec_t$ on the \emph{grandchildren} of~$t$.
This is analogous to the trees describing substitutions,
except that~$\prec_t$ is defined on the grandchildren instead of the children, hence `delayed'.
We add the technical requirement that each leaf is a~single child (with no siblings),
so that whenever~$x \neq y$ are leaves, their closest ancestor is at distance at least~2.

The \emph{realisation} of this delayed structured tree is the structure~$(L(T),<,\prec)$,
where for two leaves~$x,y$, we have~$x \prec y$ if and only if~$x' \prec_t y'$,
where~$t$ is the closest ancestor of~$x,y$, and~$x',y'$ are the grandchildren of~$t$
which are ancestors of~$x,y$ respectively.
In general, this binary relation~$\prec$ is not an order, as it might not be transitive.
We call the delayed structured tree \emph{well-formed} if~$\prec$ is a~linear order,
so that the realisation is a~permutation on~$L(T)$.
We only consider well-formed delayed structured trees.
Remark that in the realisation of the tree, we only use the linear order~$\prec_t$ between cousins;
the order between siblings is irrelevant.

We say that the permutation $(<,\prec)$ is obtained by \emph{delayed substitution}
from the permutations $\{(<_t,\prec_t)\}_{t \in T}$, understood as permutations on the grandchildren of~$t$.
If for each~$t \in T$ the permutation~$(<_t,\prec_t)$ on its grandchildren is in a~given class~$\Cc$,
we also say that the delayed structured tree~$T$ is \emph{labelled} with permutations in~$\Cc$,
and that the permutation~$(<,\prec)$ is obtained by \emph{delayed substitution from~$\Cc$}.

\subsection{Distinguishability}
The linear order~$\prec$ is defined on the leaves of~$T$.
We extend it to a~partial order on all nodes of~$T$ where~$x \prec y$ iff~$L(x) \prec L(y)$,
i.e.\ all descendants of~$x$ are strictly before all descendants of~$y$ for~$\prec$.
\begin{remark}
    \label{rmk:cousins-ordered}
    Let~$x,y$ be nodes in~$T$, and~$t$ their least common ancestor.
    If~$t$ is at distance at least~$2$ of both~$x$ and~$y$,
    then~$x,y$ are comparable by~$\prec$, meaning either~$x \prec y$ or~$y \prec x$.
    In particular, if~$x,y$ are at the same level in~$T$ but are not siblings,
    then they are comparable by~$\prec$.
\end{remark}
Being incomparable by~$\prec$ is not an equivalence relation, even when restricted to siblings:
one may have three siblings~$x,y,z$ such that~$x \prec z$,
but~$L(x),L(y)$, resp.~$L(y),L(z)$ are interleaving for~$\prec$.
However, we can define an equivalence relation by considering how other nodes can separate siblings:
let~$x,y$ be two children of a~node~$t$, and~$v \in L(T) \setminus L(t)$.
By \cref{rmk:cousins-ordered}, $v$ is comparable by~$\prec$ to both~$x$ and~$y$.
We say that~$x$ and~$y$ are \emph{distinguished} by~$v$
if $x \prec v \prec y$ or $y \prec v \prec x$.
One can check that if~$v \in L(T) \setminus L(t)$ distinguishes~$x,y$, it must be a~descendant of a~cousin of~$x,y$.
Finally, we call~$x,y$ \emph{indistinguishable}, denoted~$x \sim y$,
if~$x,y$ are siblings with parent~$t$, and no $v \in L(T) \setminus L(t)$ distinguishes~$x$ and~$y$.
See \cref{fig:distinguishability} for an example.

\begin{figure}
    \centering
    \begin{tikzpicture}[node/.style={fill,circle,inner sep=0.9mm},
                        arc/.style={very thick,-{>[length=2mm, width=2mm]}}]
        \node[node] (t) at (-1,2.7) [label=above:$t$] {};
        \foreach \i/\x in {0/-4,1/0,2/3}{
            \node[node] (c\i) at (\x,1.8) {};
            \draw[very thin] (c\i) -- (t);
        }
        \def\gr{green!50!black}
        \foreach \x/\p/\i/\c in {
            -6/c0/8/blue,-5/c0/1/green,-4/c0/3/green,-3/c0/9/blue,-2/c0/2/green,
            -1/c1/7/red,0/c1/4/yellow,1/c1/6/red,
            2/c2/10/\gr,3/c2/5/orange,4/c2/11/\gr%
        }{
            \node[node,fill=\c,draw=black] (\i) at (\x,0) {};
            \node at (\x,-1) {\i};
            \draw[very thin] (\i) -- (\p);
        }
        \foreach \i in {2,...,11}{
            \pgfmathtruncatemacro\j{\i-1}
            \draw[gray,arc] (\j) edge [bend left=30] (\i);
        }
    \end{tikzpicture}
    \caption{%
        Indistinguishability for the grandchildren of a~node~$t$.
        The values and arrows represent the linear order~$\prec_t$,
        and colours represent equivalence classes of indistinguishability.
    }
    \label{fig:distinguishability}
\end{figure}
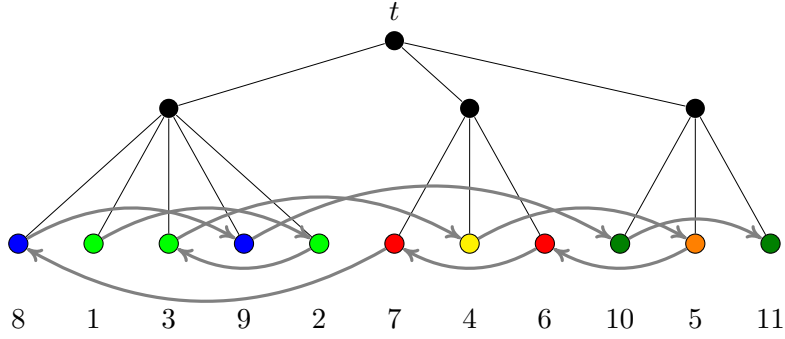

\begin{lemma}
    \label{lem:indist-equiv}
    Indistinguishability is an equivalence relation.
\end{lemma}
\begin{proof}
    From its definition, one can see that $\sim$ is reflexive and symmetric.
    Let~$x,y,z$ be siblings with parent~$t$, and suppose that~$v$ distinguishes~$x,z$, say $x \prec v \prec z$.
    Since $v \not\in L(t)$, $y$ and~$v$ are comparable.
    Thus it either holds that~$v \prec y$, in which case~$v$ distinguishes~$x,y$,
    or~$y \prec v$, and~$v$ distinguishes~$y,z$.
    By contraposition, $\sim$ is transitive.
\end{proof}
\begin{lemma}
    \label{lem:indist-realisation}
    Let~$x,y$ be grandchildren of~$t$ such that~$x \not\sim y$.
    Then $x \prec y$ if and only if $x \prec_t y$.
\end{lemma}
\begin{proof}
    Between cousins with grandparent~$t$, the orders~$\prec$ and~$\prec_t$ coincide by definition of the realisation.
    Thus, consider~$x,y$ siblings distinguished by some~$v$, say~$x \prec v \prec y$.
    Let~$v'$ be the cousin of~$x,y$ that is the ancestor of~$v$.
    Now~$x$ and~$v'$ are comparable by~$\prec$, and~$v' \prec x$ is impossible as it would imply~$v \prec x$,
    hence we have~$x \prec v'$. As noted at the beginning of the proof, this implies $x \prec_t v'$.
    The same reasoning gives~$v' \prec_t y$, hence~$x \prec_t y$,
    i.e.~$\prec$ and~$\prec_t$ coincide on the pair~$x,y$ as desired.
\end{proof}

A crucial property of indistinguishability is that it is compatible with~$\prec$ as follows.
\begin{lemma}
    \label{lem:indist-intervals}
    Let~$A$ be an equivalence class of~$\sim$.
    Then, among the leaves~$L(T)$, the subset $L(A) := \bigcup_{x \in A} L(x)$ is an interval for~$\prec$.
\end{lemma}
\begin{proof}
    Let~$A$ be an equivalence class of~$\sim$.
    All elements of~$A$ are siblings. Consider~$t$ their parent.
    Let~$x,z \in L(A)$ and~$y \in L(T) \setminus L(A)$ be leaves,
    and suppose for a~contradiction that~$x \prec y \prec z$.
    Let~$x',z'$ be the ancestors in~$A$ of~$x,z$ respectively.
    We have two cases to consider.
    \begin{enumerate}
        \item If~$y$ is not a~descendant of~$t$, then~$x'$ and~$z'$ are comparable by~$\prec$ with~$y$,
        and it must be that~$x' \prec y \prec z'$, contradicting that~$x',z'$ are indistinguishable.
        \item Otherwise, $y$ is a~descendant of some child~$y'$ of~$t$,
        which is distinguishable from~$x'$ and from~$z'$.
        Thus there exist~$v_1, v_2 \not\in L(t)$ such that~$x' \prec v_1 \prec y'$ and $y' \prec v_2 \prec z'$.
        But then we also have~$x' \prec v_1 \prec z'$, a~contradiction.
        \qedhere
    \end{enumerate}
\end{proof}

\subsection{Factoring delayed substitutions}
We will now prove that any delayed substitution can be decomposed into products of substitutions.

\begin{lemma}
    \label{lem:delayed-decomp}
    Let~$\Cc$ be a~class of permutations closed under substitution, taking patterns and inverse.
    Then any permutation obtained by delayed substitution from~$\Cc$ is in $\Cc^3$.
\end{lemma}
\begin{proof}
    Let~$(X,<,\prec)$ be a~biorder, realised by
    a~delayed structured tree~$(T,<,\{\prec_t\}_{t \in T})$ with leaves~$L(T) = X$, labelled in~$\Cc$.
    The linear order~$<$ is compatible with~$T$, while~$\prec$ is realised by the delayed substitution.

    Let us define an intermediate linear order~$<'$ on~$X$.
    For each internal node~$t$ of~$T$, and for each child~$x$ of~$t$, 
    choose an arbitrary descendant $f_t(x) \in X$ of~$x$.
    Then, on the children of~$t$, define $x <'_t y$ if and only if $f_t(x) \prec f_t(y)$.
    These local orders extend to a~linear order~$<'$ on~$X$, which by construction is compatible with~$T$.
    \begin{claim}
        For each~$t \in T$, the permutation~$(<,<'_t)$ on the children of~$t$ is in~$\Cc$.
    \end{claim}
    \begin{claimproof}
        Let~$A$ be the set of children of~$t$.
        For each~$x \in A$, let~$g(x)$ be the sole child of~$x$ which is an ancestor of~$f_t(x)$.
        Thus~$g(x)$ is a~grandchild of~$t$.
        For~$x \neq y \in A$, we have~$g(x) < g(y)$ if and only if~$x < y$, by compatibility of~$<$ with~$T$. 
        Furthermore, the closest ancestor of~$f_t(x), f_t(y)$ is $t$, 
        hence~$f_t(x) \prec f_t(y)$ if and only if $g(x) \prec_t g(y)$, by definition of~$\prec$.
        Thus the permutation~$(<,<'_t)$ on~$A$ is equal to the permutation~$(<,\prec_t)$ on~$g(A)$,
        which is a~pattern of the permutation~$(<,\prec_t)$ on all grandchildren of~$t$.
        The latter is in~$\Cc$ by hypothesis.
    \end{claimproof}
    Since~$\Cc$ is closed under substitution, it follows that the permutation~$(X,<,<')$ is in~$\Cc$.

    We now quotient the tree~$T$ by indistinguishability.
    Since~$\sim$ can only identify siblings, the quotient~$T' = T / \sim$ naturally has a~tree structure.
    Furthermore, since the leaves of~$T$ are required to be single children,
    no two leaves are identified, hence the set of leaves of~$T'$ is exactly~$X$.
    \begin{claim}
        \label{clm:intervals-quotient-tree}
        The tree~$T'$ is compatible with both~$<'$ and~$\prec$.
    \end{claim}
    \begin{claimproof}
        \Cref{lem:indist-intervals} precisely proves that~$T'$ is compatible with~$\prec$:
        each node of~$T'$, which is an equivalence class of~$\sim$, corresponds to an interval of~$(X,\prec)$.

        Since~$<'$ is compatible with~$T$, to show that it is also compatible with~$T' = T/\sim$,
        it is sufficient to consider only the children of a~given~$t \in T$.
        That is, it is enough to prove that among the children of~$t$,
        any equivalence class~$A$ for indistinguishability is an interval for~$<'$.
        Suppose for a~contradiction that there are children~$x <' y <' z$ of~$t$
        with~$x,z \in A$ and~$y \not\in A$.
        Since $x \not\sim y$, it must be that~$x \prec y$ or~$y \prec x$,
        but the latter is impossible as it would imply $y <' x$.
        Thus $x \prec y$, and similarly~$y \prec z$, contradicting \cref{lem:indist-intervals}.
    \end{claimproof}
    Thus the permutation~$(X,<',\prec)$ is obtained by substitution along the tree~$T'$.
    It only remains to show that this structured tree is labelled with permutations in~$\Cc^2$.
    
    Fix a~node~$\bar{s} \in T'$, which in~$T$ is an equivalence class of~$\sim$.
    Let~$t \in T$ be the parent of the nodes of~$\bar{s}$
    (only siblings can be in the same equivalence class for~$\sim$).
    The children in~$T'$ of~$\bar{s}$ are equivalence classes of grandchildren of~$t$ in~$T$.
    Denote by~$\bar{x}_1,\dots,\bar{x}_k$ these children,
    where the representatives~$R = \{x_1,\dots,x_k\}$ are grandchildren of~$t$.
    Recall that~$L(x_i) \subseteq X$ is the set of leaves descendant of~$x_i$.
    The subsets~$L(x_1),\dots,L(x_k)$ are non-interleaved
    for $<$ and~$<'$ because these linear orders are compatible with~$T$.
    Furthermore, they are non-interleaved for~$\prec$ because the~$x_i$ are pairwise distinguished.
    Thus~$R$ is equipped with the three linear orders~$<,<',\prec$.
    In the structured tree~$T'$ which realises~$(X,<',\prec)$,
    $\bar{s}$ is labelled with~$(R,<',\prec)$,
    hence we only need the following to conclude.
    \begin{claim}
        The permutation $(R,<',\prec)$ is in~$\Cc^2$.
    \end{claim}
    \begin{claimproof}
        We decompose the permutation~$(R,<',\prec)$ into~$(R,<',<)$ and~$(R,<,\prec)$.
        This may at first seem counterproductive, as we go back to the initial permutation~$(<,\prec)$,
        but the point is that we now only consider a~small subset~$R$, rather than all~$X$.

        We already know that~$(X,<,<')$ is in~$\Cc$,
        hence so is~$(R,<',<)$, which is a~pattern of its inverse.
        Furthermore, the~$x_i$ are pairwise distinguished grandchildren of~$t$,
        hence the linear orders~$\prec$ and~$\prec_t$ coincide on~$R$ by \cref{lem:indist-realisation}.
        Thus $(R,<,\prec)$ is a~pattern of~$(<,\prec_t)$, which is in~$\Cc$.
    \end{claimproof}
    As $(X,<,\prec)$ can be factorised into $(X,<,<') \in \Cc$ and $(X,<',\prec) \in \Cc^2$, it is in $\Cc^3$.
\end{proof}

\subsection{Constructing delayed structured trees}
Finally, we show how to express any permutation as a~delayed substitution.
\begin{lemma}
    \label{lem:delayed-construction}
    Let~$\sigma = (X,<,\prec)$ be a~biorder.
    There is a~structured tree~$(T,<,\{\prec_t\}_{t \in T})$ whose realisation is~$\sigma$,
    that satisfies the following.
    \begin{enumerate}
        \item \label{cond:order-representatives}
        For any node~$t \in T$, the linear order~$\prec_t$ on its grandchildren
        is obtained by a~choice of representatives.
        That is, if~$A$ is the set of grandchildren of~$t$,
        then there is a~mapping $f : a~\to X$, with~$f(a)$ descendant of~$a$,
        such that for~$a,b \in A$, $a \prec_t b$ if and only if $f(a) \prec f(b)$.
        \item \label{cond:siblings-split}
        If~$x,y$ are consecutive siblings along~$<$,
        then~$x$ and~$y$ are distinguished for the linear order~$\prec$,
        except possibly if~$x$ and~$y$ do not have any other sibling.
    \end{enumerate}
\end{lemma}
In the previous statement, Condition~\ref{cond:order-representatives} is a~technical requirement
ensuring that the permutations labelling~$T$ are patterns of~$\sigma$,
so that hypotheses on~$\sigma$ can also be applied to the former.
Condition~\ref{cond:siblings-split} is crucial for the induction on almost mixed minors in \cref{sec:decomposition}.
Informally, it ensures that the subpermutations $\{(<,\prec_t)\}_{t \in T}$ labelling~$T$
are strictly simpler---in the sense of almost mixed minors---than the global permutation~$(X,<,\prec)$.
\begin{proof}[Proof of \cref{lem:delayed-construction}]
    We construct the tree~$T$ inductively starting from the root,
    choosing for each internal node~$t$ the interval (for~$<$)~$L(t) \subseteq X$
    that, at the end of the construction, will be the set of leaves descendant of~$t$.
    We maintain the condition that whenever~$t'$ is a~child of~$t$ and~$x \in X \setminus L(t)$,
    then~$x$ does not \emph{split}~$L(t')$, i.e.\ either~$x \prec L(t')$ or~$x \succ L(t')$.

    Initially, there is only the root~$r$ with~$L(r) = X$.
    If~$t$ and~$L(t)$ have already been constructed,
    then we create the children of~$t$ by the following rules.
    \begin{itemize}
        \item If~$L(t) = \{x\}$ is a~singleton, we add a~leaf~$x$ as the sole child of~$t$,
        and the construction stops there for this branch.
        
        \item If~$L(t)$ has size at least~2, and is an interval of~$(X,\prec)$
        (in addition to being one for $<$),
        then we add two children~$u,v$ to~$t$,
        and split~$L(t)$ arbitrarily into two intervals~$L(u),L(v)$ for~$<$.
        Remark that this case occurs at least for the root~$r$.
        
        If~$x \not\in L(t)$, then~$x$ does not split~$L(t)$,
        hence a~fortiori it splits neither~$L(u)$ nor~$L(v)$, as required.
        On the other hand, this means that~$u$ and~$v$ will be indistinguishable in~$T$,
        which is why condition~\ref{cond:siblings-split}
        is waived for nodes with only two children.
        
        \item Otherwise, enumerate~$L(t)$ as~$x_1 < \dots < x_k$.
        Say that a~subset~$A \subseteq L(t)$ is a~\emph{local module} of~$L(t)$
        if no~$y \not\in L(t)$ splits~$A$ (for~$\prec$).
        We greedily partition~$L(t)$ into local modules~$A_1,\dots,A_l$:
        $A_1$ is~$\{x_1,\dots,x_{i_1}\}$ with~$i_1$ chosen maximal such that~$A_1$ is a~local module,
        then~$A_2$ is~$\{x_{i_1 + 1},$ $\dots,x_{i_2}\}$ with again~$i_2$ maximal, etc.
        
        By construction, no element~$y \not\in L(t)$ can split any~$A_i$.
        Further, $A_i,A_{i+1}$ are distinguished
        (in fact, the last element of~$A_i$ is distinguished from the first element of~$A_{i+1}$),
        as otherwise~$A_i$ could have been extended.
        Also, the number~$l$ of local modules in this partition is at least~2,
        as otherwise~$L(t)$ would be an interval of~$(X,\prec)$ and we would fall in the previous case.
    \end{itemize}

    With this tree constructed,
    consider a~node~$t \in T$, and $u,v$ two grandchildren of~$t$ that are not siblings.
    Then the condition maintained during the construction ensures that
    no~$x \in L(u)$ splits~$L(v)$, and symmetrically no~$y \in L(v)$ splits~$L(u)$.
    It follows that either~$L(u) \prec L(v)$ or~$L(v) \prec L(u)$.
    In short, the order~$\prec$ is well defined between cousins in~$T$.

    Now, define the linear order~$\prec_t$ on the grandchildren through an arbitrary choice of representatives,
    as required in condition~\ref{cond:order-representatives} of the lemma.
    Then for cousins~$u,v$, we have~$u \prec_t v$ if and only if~$u \prec v$.
    In turn, this implies that the realisation of~$(T,<,\{\prec_t\}_{t \in T})$
    is the permutation~$(X,<,\prec)$ from which we started.
    Condition~\ref{cond:order-representatives} is satisfied by definition of~$\prec_t$,
    and condition~\ref{cond:siblings-split} was verified during the construction of~$T$.
\end{proof}

\section{Partitions and mixity}\label{sec:quotients}
In this section, we introduce tools to further decompose a~single level of a~delayed structured tree.
Their purpose is similar to the right module partitions of~\cite[Section~4]{tww-poly-chi-bounded},
but once again working with biorders gives stronger results and simpler proofs.

\subsection{Definitions}\label{sec:split-forest}

Let~$(X,\prec)$ be a~linear order, and~$\Pc$ a~partition of~$X$.
For any subsets~$X_1,X_2 \in \Pc$, we distinguish three cases:
\begin{enumerate}
    \item $X_1,X_2$ do not interleave, i.e.~$X_1 \prec X_2$ or~$X_2 \prec X_1$.
    \item Restricted to $X_1 \cup X_2$, $X_1$ is an interval
    which is said to \emph{split}~$X_2$ in two, written~$X_1 \sqsubset X_2$, or vice versa.
    \item Neither~$X_1$ nor~$X_2$ is an interval in~$X_1 \cup X_2$,
    and we say that~$X_1,X_2$ are \emph{mixed}.
\end{enumerate}
These three cases correspond to zones in the adjacency matrix that are respectively
constant, vertical or horizontal, and mixed.

Assume now that no pair of subsets in~$\Pc$ is mixed---we then say that~$\Pc$ is \emph{non-mixed}.
For any~$X_1 \in \Pc$, let~$\clos{X}_1$ denote the interval closure of~$X_1$,
i.e.\ the interval of~$(X,\prec)$ between the minimum and maximum of~$X_1$.
Then we have~$X_1 \prec X_2$ if and only if $\clos{X}_1 \prec \clos{X}_2$,
and~$X_1 \sqsubset X_2$ if and only if $\clos{X}_1 \subset \clos{X}_2$.
In particular, either~$\clos{X}_1,\clos{X}_2$ are disjoint, or one contains the other.
Thus~$(X,\Pc)$ induces a~laminar family, meaning that subsets in~$\Pc$ are the nodes of a~rooted forest~$F$,
where~$X_1$ is a~descendant of~$X_2$ if and only if $X_1 \sqsubset X_2$.
Furthermore, the order~$\prec$ is defined for any pair of subsets~$X_1,X_2 \in \Pc$
that are not in an~ancestor--descendant relationship.
Thus~$\prec$ gives to~$F$ the structure of an ordered forest:
each connected component of~$F$ is an ordered tree, and furthermore~$\prec$ orders the components.

\subsection{Non-mixed partitions}
We now consider a~biorder~$(X,<,\prec)$, and a~partition~$\Pc$ of~$X$ into intervals for~$<$.
In later proofs, such a~structure will arise from delayed substitutions as follows:
Given a~node~$t$ in a~delayed structured tree, $X$ is the set of grandchildren of~$t$,
the partition~$\Pc$ is the one given by the children of~$t$,~$<$ is the linear order compatible with the tree, while~$\prec$
is the linear order on the grandchildren of~$t$ given as part of the delayed structured tree.
Our goal is to decompose this biorder into simpler permutations.

\begin{lemma}\label{lem:non-mixed-decomp}
    Let $\mathcal{C}$ be a~substitution-closed class of permutations.
    Let~$(X,<,\prec)$ be a~biorder, and~$\Pc$ a~partition of~$X$ into intervals of~$<$.
    Suppose that~$\Pc$ is non-mixed w.r.t.~$\prec$.
    Finally, assume that
    \begin{enumerate}
        \item for each part~$P \in \Pc$, $(P,<,\prec)$ is in~$\Cc$, and
        \item there exists a~\emph{transversal}~$R$ of~$\Pc$
        (i.e.\ a~choice of a~single element in each part of~$\Pc$)
        such that the permutation~$(R,<,\prec)$ is in~$\Cc$.
    \end{enumerate}
    Then~$(X,<,\prec)$ is in~$\Sep^2 \circ \Cc$, where~$\Sep$ denotes the class of separable permutations.
\end{lemma}
\begin{proof}
    Let us define an intermediate linear order~$<'$ on~$X$ as follows:
    \begin{itemize}
        \item each part~$P \in \Pc$ is an interval of~$(X,<')$,
        \item inside each part~$P$, the linear order~$<'$ coincides with~$\prec$, and
        \item between parts~$P_1,P_2 \in \Pc$, the order is given by the representatives in~$R$,
        i.e.\ if~$x_i$ is the sole element in~$P_i \cap R$, then~$P_1 <' P_2$ if and only if~$x_1 \prec x_2$.
    \end{itemize}
    \begin{claim}
        The permutation~$(X,<,<')$ is in~$\Cc$.
    \end{claim}
    \begin{claimproof}
        The permutation inside a~given part~$P \in \Pc$ is $(P,<,\prec)$, which is known to be in~$\Cc$.
        Furthermore, the permutation on the set~$\Pc$ is isomorphic to~$(R,<,\prec)$, which is also in~$\Cc$.
        Therefore $(X,<,<')$ can be expressed as a~substitution from permutations in~$\Cc$,
        with a~structured tree of depth~2.
    \end{claimproof}
    We now focus on~$(X,<',\prec)$,
    and will prove that this permutation is obtained by substitution from 2-shuffles.
    Such permutations are in~$\Sep^2$ by \cref{lem:shuffle-decomp} and since~$\Sep$ is closed under substitution.

    Recall the structure on~$(X,\prec)$ and~$\Pc$ described in \cref{sec:split-forest}:
    Since the partition~$\Pc$ is non-mixed w.r.t.~$\prec$,
    the \emph{splitting} partial order~$\sqsubset$ (related to $\prec$) gives~$\Pc$ the structure of a~rooted forest~$F$,
    which furthermore is compatible with the linear order~$\prec$.
    We construct a~tree~$T$ from~$F$ by
    \begin{itemize}
        \item adding a~new root, whose children are the roots of the connected components of~$F$, and
        \item for each~$P \in \Pc$ node of~$F$, adding each element~$x \in P$ as a~new child of~$P$.
    \end{itemize}
    The leaves of~$T$ are exactly the elements of~$X$, and we will describe $(X,<',\prec)$ as a~substitution along~$T$. See \cref{fig:split-tree} for an illustration.
    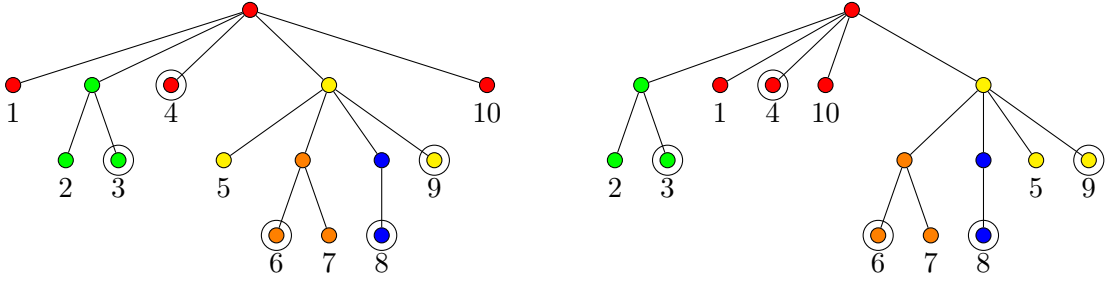
\begin{figure}
        \centering
        \begin{tikzpicture}
            \foreach \n/\x/\z/\y/\p/\c/\l in {
                x0/0/0/4/x0/red/,y0/-4.5/-2.5/3/x0/red/1,y3/-1.5/-1.5/3/x0/red/4,y9/4.5/-0.5/3/x0/red/10,
                x1/-3/-4/3/x0/green/,y1/-3.5/-4.5/2/x1/green/2,y2/-2.5/-3.5/2/x1/green/3,
                x2/1.5/2.5/3/x0/yellow/,y4/-0.5/3.5/2/x2/yellow/5,y8/3.5/4.5/2/x2/yellow/9,
                x3/1/1/2/x2/orange/,y5/0.5/0.5/1/x3/orange/6,y6/1.5/1.5/1/x3/orange/7,
                x4/2.5/2.5/2/x2/blue/,y7/2.5/2.5/1/x4/blue/8%
            }{
                \node[vertex,fill=\c,minimum width=2mm] (\n1) at (\x*0.7,\y) [label=below:\l] {};
                \node[vertex,fill=\c,minimum width=2mm] (\n2) at (\z*0.7 + 8,\y) [label=below:\l] {};
                \ifx\n\p\else
                \draw (\n1) -- (\p1);
                \draw (\n2) -- (\p2);
                \fi
            }
            \foreach \n in {y2,y3,y5,y7,y8}{
                \draw (\n1) circle (2mm);
                \draw (\n2) circle (2mm);
            }

        \end{tikzpicture}
        \caption{%
            The tree~$T$, ordered by~$\prec$ (left) and~$<'$ (right) respectively.
            Leaves are elements of~$X$, and their numbering is according to~$\prec$.
            Internal nodes correspond to parts in~$\Pc$, with each colour representing a~part of~$\Pc$.
            The transversal~$R$ is indicated by circled leaves, and decides how~$<'$ orders the parts.
        }
        \label{fig:split-tree}
    \end{figure}

    \begin{claim}
        The linear orders~$<'$ and~$\prec$ are compatible with~$T$.
    \end{claim}
    \begin{claimproof}
        Consider a~node $t$ of~$T$, other than the leaves and the root for which it is straightforward that $L(t)$ is an interval of $<'$ and of $\prec$.
        Thus $t$ corresponds to some part~$P \in \Pc$ (and we now identify $t$ to $P$).
        Then it is simple to check that~$L(P)$, the set of leaves descending from~$P$ in~$T$,
        is exactly~$\bigcup_{P' \sqsubseteq P} P'$,
        i.e.\ the union of parts which are descendants of~$P$ in~$F$.
        We claim that this is an interval for both~$<'$ and~$\prec$.

        Suppose otherwise towards a~contradiction.
        Then there are~$u,w \in L(P)$ and~$v \not\in L(P)$ such that
        either~$u \prec v \prec w$ or~$u <' v <' w$.
        Let~$U,V,W$ be the parts of~$\Pc$ containing~$u,v,w$ respectively.
        We have~$U,W \sqsubseteq P$, and~$V \not\sqsubseteq P$.
        Recall that~$\clos{P}$ denotes the interval between the minimum and maximum of~$P$ for~$\prec$.
        Then we have~$U,W \subseteq \clos{P}$, while~$V$ is disjoint from~$\clos{P}$.
        This implies~$u,w \in \clos{P}$ but~$v \not\in \clos{P}$,
        hence it cannot be that~$u \prec v \prec w$ since~$\clos{P}$ is an interval of~$\prec$.

        For~$<'$, we also need to consider the representative~$v' \in V \cap R$.
        For the same reasons as above, we have either~$v' \prec U \cup W$, or~$U \cup W \prec v'$,
        which implies~$V <' U \cup W$ or~$U \cup W <' V$ respectively.
        Either way, $u <' v <' w$ is impossible.
    \end{claimproof}

    To conclude the proof, we only need to show that for any internal node~$P$ of~$T$,
    the permutation $(X,<',\prec)$ restricted to its children is a~2-shuffle.
    We partition the children of~$P$ in two categories:
    leaves (which are in~$X$), and internal nodes (which are in~$\Pc$).
    We claim that~$(X,<',\prec)$ restricted to either of these categories is the identity.
    \begin{enumerate}
        \item The leaves that are children of~$P$ are exactly the elements of~$P$,
        and inside~$P$ the linear orders~$<'$ and~$\prec$ coincide by construction.
        \item The parts of~$\Pc$ that are children of~$P$ in~$T$ are exactly the children of~$P$ in~$F$.
        We know that they are totally ordered under~$\prec$,
        i.e.\ if~$A,B \in \Pc$ are children of~$P$, then either~$A \prec B$ or~$B \prec A$.
        By construction of~$<'$, if~$A \prec B$ (resp.~$B \prec A$) then~$A <' B$ (resp.~$B <' A$).
        It follows from these two remarks that~$<'$ and~$\prec$ coincide on the children of~$P$ in~$F$.
        \qedhere
    \end{enumerate}
\end{proof}

\subsection{Separating mixed parts}

The previous lemma shows how to decompose a~permutation that is non-mixed w.r.t.\ a~given partition.
We will now generalise it to permutations with few pairs of mixed parts.
Let~$(X,<)$ be a~linear order and~$\Pc$ a~partition of~$X$.
The associated \emph{mixed graph} is the graph with vertex set~$\Pc$,
and where two parts~$A,B \in \Pc$ are adjacent when they are mixed.

\begin{lemma}
    \label{lem:mixed-degenerate}
    Let~$G$ be the mixed graph of a~partition,
    and assume that it is $K_t$-free and $K_{t,t}$-subgraph-free.
    Then~$G$ is~$(4t^2-1)$-degenerate.
\end{lemma}
\begin{proof}
    Let~$G = (\Pc,E)$ be the mixed graph of a~partition~$\Pc$ of a~linear order~$(X,<)$.
    We partition the edges of~$G$ into~$E = E_1 \uplus E_2$ as follows:
    \begin{enumerate}
        \item An edge~$AB \in E$ is in~$E_1$ if the interval closures satisfy
        $\clos{A} \subset \clos{B}$ or~$\clos{B} \subset \clos{A}$.
        \item The remaining edges are in~$E_2$.
        Thus, a~pair~$(A,B) \in \Pc$ is an edge in~$E_2$ if and only if~$\clos{A}$ and~$\clos{B}$ overlap.
    \end{enumerate}
    The graph~$(\Pc,E_2)$ is precisely the overlap graph of the interval closures of parts of~$\Pc$.
    Thus by \cref{lem:circle-graph-degenerate}, its maximum edge density is at most~$2(t-1)^2$.
    We will show that~$(\Pc,E_1)$ is $(t-1)$-degenerate, hence has maximum edge density at most~$t-1$.
    Then, their union~$(\Pc,E)$ has maximum edge density at most~$t-1 + 2(t-1)^2 \le 2t^2-1$,
    hence it is~$(4t^2 - 1$)-degenerate as desired.

    To this end, we define an acyclic orientation of~$(\Pc,E_1)$ for which the out-degree is at most~$t-1$.
    If~$AB$ is an edge in~$E_1$ and~$\clos{A} \subset \clos{B}$, we orient it from~$A$ to~$B$.
    We claim that for any part~$A \in \Pc$, the out-degree of~$A$ in~$(\Pc,E_1)$ is at most~$t-1$.
    Indeed, let~$B_1,\dots,B_k$ be the out-neighbours of~$A$,
    so that each~$B_i$ is mixed with~$A$, and~$\clos{A} \subset \clos{B}_i$ (see \cref{fig:mixed-intervals}).
    We will show that~$A,B_1,\dots,B_k$ is a~clique in~$(\Pc,E)$,
    which implies the claim as this graph has no clique of size~$t$.
    To this end, we need to show that any two~$B_i \neq B_j$ are mixed.
    Without loss of generality, assume that~$\clos{B}_i \not\subset \clos{B}_j$,
    hence there is some~$x \in B_i$ outside of~$\clos{B}_j$, say~$x < B_j$.
    On the other hand, since~$B_i$ is mixed with~$A$,
    there must be~$y \in B_i \cap \clos{A}$, and thus $y \in B_i \cap \clos{B}_j$.
    Thus we have two points in~$B_i$, one inside the interval closure of~$B_j$ and one outside.
    This implies that~$B_i$ and~$B_j$ are mixed.
    \begin{figure}
        \centering
        \begin{tikzpicture}
            \def\xs{0.7}
            \def\ys{0.4}
            \foreach \a/\b/\y/\c in {
                8/13/1/black,7/18/2/gray,1/15/3/gray,2/19/4/gray,3/16/0/gray,0/11/-1/gray%
            }{
                \draw[very thick,\c] (\a*\xs,\y*\ys) -- (\b*\xs,\y*\ys);
                \node[vertex] at (\a*\xs,\y*\ys) {};
                \node[vertex] at (\b*\xs,\y*\ys) {};
            }

            \foreach \x/\y in {
                4/4,5/0,6/-1,9/4,10/2,12/3,14/0,17/4%
            }{
                \node[vertex] at (\x*\xs,\y*\ys) {};
            }

            \foreach \i/\n in {1/$A$,2/$B_1$,3/$B_2$,4/$B_3$,0/$C$,-1/$D$}{
                \node at (-1,\i*\ys) {\n};
            }

            \foreach \i/\x in {1/9,2/12,3/10}{
                \node (y\i) at (\x*\xs,5*\ys) {$y_\i$};
                \draw[dashed] (y\i) -- +(0,-4*\ys);
            }
        \end{tikzpicture}
        \caption{
            Representation of six parts of~$\Pc$.
            Each row is a~part~$P$, the dots being elements of~$P$, and the line being the interval~$\clos{P}$.
            The left-to-right order is~$\prec$.
            The~$B_i$s are out-neighbours of~$A$ in the graph~$(\Pc,E_1)$,
            and~$y_i \in B_i \cap \clos{A}$ witnesses that~$A$ and~$B_i$ are mixed.
            For any~$i \neq j$, $y_i,y_j$ together with the endpoints of~$\clos{B_i},\clos{B_j}$
            ensure that~$B_i$ and~$B_j$ are mixed.
            The part~$C$ satisfies~$\clos{A} \subset \clos{C}$, but is not mixed with~$A$ as~$C \cap \clos{A} = \emptyset$.
            The part~$D$ is mixed with~$A$ because~$\clos{A}$ and~$\clos{D}$ overlap, hence~$AD \in E_2$.
        }
        \label{fig:mixed-intervals}
    \end{figure}
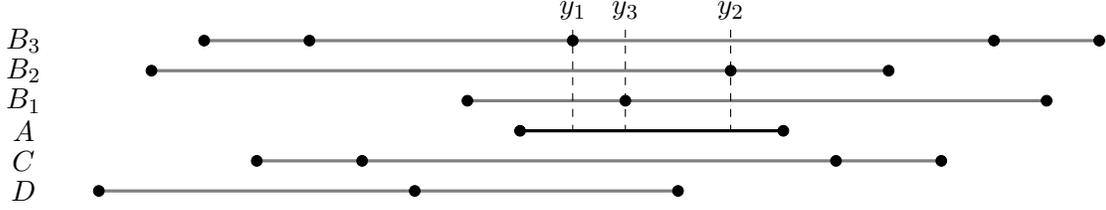
\end{proof}

Combining this with \cref{lem:non-mixed-decomp}, we obtain the following, where (as throughout the paper) the $\log$ function is in base~2. 
\begin{lemma}
    \label{lem:grandchildren-decomp}
    Let~$\Cc$ be a~class of permutations closed under substitution, symmetry, and taking patterns.
    Let~$(X,<,\prec)$ be a~biorder, and~$\Pc$ a~partition of~$X$ into intervals of~$<$,
    satisfying the following:
    \begin{enumerate}
        \item The mixed graph of~$\Pc$ for the linear order~$(X,\prec)$
        does not contain~$K_t$ or~$K_{t,t}$ subgraphs.
        \item For each part~$P \in \Pc$, the permutation~$(P,<,\prec)$ is in~$\Cc$.
        \item There exists a~transversal~$R$ of~$\Pc$
        (i.e.\ a~choice of a~single element in each part of~$\Pc$)
        such that the permutation~$(R,<,\prec)$ is in~$\Cc$.
    \end{enumerate}
    Then the permutation~$(X,<,\prec)$ is in $\Sep^{k+2} \circ \Cc \circ \Sep^k$,
    with~$k = \ceil{2\log t} + 2$.
\end{lemma}
\begin{proof}
    Using \cref{lem:mixed-degenerate,lem:degen-colouring}, the mixed graph of~$\Pc$ is $4t^2$-colourable.
    Fix a~proper $4t^2$-colouring,
    and for any fixed colour~$c \in [4t^2]$, let~$\Pc_c$ be the set of parts of colour~$c$,
    and~$X_c = \bigcup_{P \in \Pc_c} P$ the points contained therein.
    Since the colouring is proper, no two parts of~$\Pc_c$ are mixed.
    Thus the restricted biorder~$(X_c,<,\prec)$ with the partition~$\Pc_c$
    satisfies the conditions of \cref{lem:non-mixed-decomp},
    and the permutation~$(X_c,<,\prec)$ is in~$\Sep^2 \circ \Cc$.
    Hence, $(X,<,\prec)$ is a~$4t^2$-shuffle of permutations in~$\Sep^2 \circ \Cc$.
    We conclude by applying \cref{lem:shuffle-decomp}
    with $k = \ceil{\log(4t^2)} = \ceil{2\log t} + 2$.
\end{proof}

\section{Reducing the size of almost mixed minors}\label{sec:decomposition}
Given a~biorder~$\sigma = (X,<,\prec)$ define its \emph{adjacency matrix}~$M_\sigma$
as the 0,1-matrix whose rows and columns are~$X$ ordered by~$<$,
and with a~1 in position~$(x,y)$ if and only if~$x \prec y$, 
i.e. if and only if~$\sigma(x) < \sigma(y)$.
Equivalently, $M_\sigma$ is obtained by applying~$\sigma^{-1}$
to both the rows and the columns of the full upper triangular matrix.
This should not be confused with the more usual sparse permutation matrix,
with a~1 in positions of the form~$(x,\sigma(x))$.
We say that~$\sigma$ is $k$-almost mixed free as shorthand for~$M_\sigma$ being $k$-almost mixed free.

Using the tools from the previous sections,
we will show how to factorize $k$-almost mixed free permutations, by induction on~$k$.
The base case is the following simple remark.
\begin{lemma}
  \label{lem:2almost-mixed-free}
    Any 2-almost mixed free permutation is separable.
\end{lemma}
\begin{proof}
    A~permutation that is not separable contains either~$3142$ or~$2413$ as pattern,
    both containing 2-almost mixed minors in their adjacency matrices
    (they in fact also contain 2-mixed minors).
    See \cref{fig:mixed-nonsep} for the adjacency matrix of~$3142$, the matrix of~$2413$ is its transpose.
    \begin{figure}[ht]
      \begin{center}
        \begin{tikzpicture}[scale=0.4]
          \foreach \i/\si in {1/3,2/1,3/4,4/2}{
            \node(r\i) at (-0.6,-\i) {\si};
            \node(c\i) at (\i,0.6) {\si};

            \foreach \j/\sj in {1/3,2/1,3/4,4/2}{
              \ifnum \sj<\si
                \node(\i\j) at (\i,-\j) {1};
              \else
                \node(\i\j) at (\i,-\j) {0};
              \fi
            }
          }
          \draw (0.5,-0.5) -- (4.5,-0.5) -- (4.5,-4.5) -- (0.5,-4.5) -- (0.5,-0.5);
          \draw[dashed] (2.5,-0.5) -- (2.5,-4.5);
          \draw[dashed] (0.5,-2.5) -- (4.5,-2.5);
        \end{tikzpicture}
      \end{center}
      \caption{%
        Adjacency matrix of the permutation~$\sigma = 3142$,
        which contains a~2-mixed minor represented by dashed lines.
        The matrix has a~1 at the intersection of the $i$th row and $j$th column
        if and only if~$\sigma(i) < \sigma(j)$.
      }
      \label{fig:mixed-nonsep}
    \end{figure}
\end{proof}

Let~$\Ac_k$ denote the class of $k$-almost mixed free permutations,
and recall that~$\Sep$ is the class of separable permutations.
\begin{theorem}
    \label{thm:main-induction}
    If $\Ac_{k-1} \subseteq \Sep^r$ for some~$r \in \Nn$, then $\Ac_k \subseteq \Sep^{s}$ for
    \[ s = 3r + 12\ceil{\log k} + 28. \]
\end{theorem}
\begin{proof}
    Assume that $\Ac_{k-1} \subseteq \Sep^r$,
    and consider $\sigma = (X,<,\prec) \in \Ac_k$ a~$k$-almost-mixed free permutation.
    Denote by $\first$ and $\last$ the minimum and maximum of~$X$ w.r.t.~$<$.
    For the linear order~$\prec$, $\first$ and~$\last$ split~$X$ into three intervals as
    \[ X_1 \prec \first \prec X_2 \prec \last \prec X_3, \]
    up to swapping~$\first$ and~$\last$.
    We consider each~$X_i$ independently, before recombining~$X_1,X_2,X_3,$ and~$\{\first,\last\}$.
    Let $X' = X_i$ be one of these intervals, and~$\sigma' = (X',<,\prec)$ the restricted permutation.

    Let~$(T,<,\{\prec_t\}_{t \in T})$ be a~delayed structured tree for~$\sigma'$
    obtained by \cref{lem:delayed-construction}.
    Consider any internal node~$t \in T$, let~$A$ be its set of children,
    and consider a~transversal~$S \subseteq X'$ of~$A$,
    i.e.\ all elements of~$S$ are descendants of~$t$,
    and each child~$v \in A$ has exactly one descendant in~$S$.
    \begin{claim}
        \label{clm:children-perm}
        The transversal~$S$ admits a~partition into two parts~$S = S_L \uplus S_R$
        such that the permutation~$(<,\prec)$ restricted to each of~$S_L,S_R$ is $(k-1)$-almost mixed free.
    \end{claim}
    \begin{claimproof}
        Enumerate the children of~$t$ as~$v_1,\dots,v_\ell$ in the linear order~$<$,
        and let~$x_i$ be the descendant of~$v_i$ in~$S$.
        We can assume~$\ell > 2$ as the claim is otherwise trivial,
        hence condition~\ref{cond:siblings-split} of \cref{lem:delayed-construction}
        gives that for all~$i \in \{1,\dots,\ell-1\}$, $v_i$ and~$v_{i+1}$ are distinguishable.
        Thus, we can choose some~$y_i \in X' \setminus L(t)$
        such that~$v_i \prec y_i \prec v_{i+1}$, or~$v_{i+1} \prec y_i \prec v_i$.

        If $y_i < L(t)$ (resp.~$y_i > L(t)$) we say that~$v_i$ is \emph{split to the left} (resp.\ \emph{to the right}).
        Thus every~$v_i$ except~$v_\ell$ is split either to the left or to the right.
        We partition~$S$ accordingly, so that~$S_L$ (resp.~$S_R$) contains~$x_i$
        only if~$v_i$ is split to the left (resp.\ to the right).
        We place~$x_\ell$ in either~$S_L$ or~$S_R$ arbitrarily.
        We will prove that~$(S_L,<,\prec)$ is $(k-1)$-almost mixed free,
        the case of~$S_R$ being symmetrical.

        Suppose for a~contradiction that $(S_L,<,\prec)$ has a~$(k-1)$-almost mixed minor.
        It is given by two partitions of~$(S_L,<)$ into~$(k-1)$ intervals, say
        \[ \Rc = \{R_1 < \dots < R_{k-1}\} \qquad \text{and} \qquad \Cc = \{C_1 < \dots < C_{k-1}\} \]
        such that for every~$i \neq j \in [k-1]$, $R_i$ and~$C_j$ are mixed in the linear order~$\prec$.
        Note in particular that all~$R_i$ and~$C_j$ have size at least~2.
        Starting from~$\Rc$ and~$\Cc$, we build two new partitions~$\Rc' = \{R'_0 < \dots < R'_{k-1}\}$
        and~$\Cc' = \{C'_0 < \dots < C'_{k-1}\}$, which will form a~$k$-almost mixed minor of (a pattern of)~$\sigma$, as follows:

        \begin{itemize}
            \item Initially set~$R'_0 = \{\first\}$,~$C'_0  = \{\first\}$, and ~$R'_i = R_i$ and~$C'_i = C_i$ for~$i \ge 1$.
            \item For each~$i \ge 1$, consider~$x_j$ the minimum of~$R_i$ w.r.t.~$<$.
            Since~$R_i$ has size at least~2, we have~$j < \ell$,
            and either~$x_{j+1} \in R_i$, or~$x_{j+1}$ is not in~$S_L$.
            We add~$x_{j+1}$ to~$R'_i$ in the latter case, and in both cases we add~$y_j$ to~$C'_0$.
            Recall that $y_j < L(t)$ since $v_j$ is split left.
            \item The same operation is applied to~$\Cc'$, adding~$y_j$ to~$R'_0$, and~$x_{j+1}$ to~$C'_i$ if needed.
        \end{itemize}
        After this modification, all elements of~$R'_0$ are smaller than~$L(t)$ for~$<$,
        hence~$\Rc'$ is still a~partition of a~subset of~$X$ into intervals of~$<$, and similarly with~$\Cc'$.
        The parts which were originally in~$\Rc,\Cc$ have only increased,
        hence for~$i \neq j$, $i,j \ge 1$, $R'_i$ and~$C'_j$ are mixed.
        We claim that $R'_0$ is mixed with~$C'_i$ for any~$i > 0$, and symmetrically for~$C'_0,R'_i$,
        which implies that the new~$\Rc',\Cc'$ form a~$k$-almost mixed minor in (a submatrix of)~$M_{\sigma}$, a~contradiction.
        
        Indeed, let~$x_j$ be the smallest element of~$R_i$,
        so that~$y_j \in C'_0$ distinguishes~$x_j$ and~$x_{j+1}$,
        i.e.\ either $x_j \prec y_j \prec x_{j+1}$ or $x_{j+1} \prec y_j \prec x_j$.
        On the other hand, by construction of~$X'$,
        $\first$ is not interleaved with~$x_j,x_{j+1}$ (or any elements of~$X'$),
        i.e.\ either~$\first \prec x_j,x_{j+1}$ or $x_j,x_{j+1} \prec \first$.
        The above implies that $\{\first,y_j\}$ and~$\{x_j,x_{j+1}\}$ are mixed.
        The former is contained in~$C'_0$, while the latter is contained in~$R'_i$,
        proving the claim.

        The same reasoning also holds for~$S_R$, using~$\last$ instead of~$\first$ and adding parts~$R'_t$ and~$C'_t$ instead of~$R'_0$ and~$C'_0$.
    \end{claimproof}
    It follows from \cref{clm:children-perm} that for any transversal~$S$ of the children of a~node~$t$,
    the permutation $(S,<,\prec)$ is a~2-shuffle of permutations in~$\Ac_{k-1}$,
    hence is in~$\Sep^{r+2}$ by \cref{lem:shuffle-decomp}.

    Fix now a~node~$t \in T$, and let~$A$ be the set of grandchildren of~$t$.
    Consider the permutation $(A,<,\prec_t)$, and the partition~$\Pc$ of~$A$ induced by the children of~$t$.
    Recall that $\Sep^{r+2}$ is closed under substitution, inverse and taking patterns.
    Let us check that the three conditions of \cref{lem:grandchildren-decomp} are satisfied.
    \begin{enumerate}
        \item If the mixed graph of~$\Pc$ for the linear order~$(A,\prec_t)$ contains either~$K_k$ or~$K_{k,k}$ as a~subgraph,
        then the corresponding parts form a~$k$-almost mixed minor in the matrix of~$(A,\prec_t)$ ordered by~$<$.
        We know that the linear order~$\prec_t$ is defined through a~choice of representatives,
        hence~$(A,<,\prec_t)$ is a~pattern of~$(X,<,\prec)$,
        meaning that the latter would also have a~$k$-almost mixed minor, a~contradiction.
        \item Fix a~part~$P \in \Pc$. The elements of~$P$ are the children of some child~$t'$ of~$t$,
        thus the restricted permutation~$(P,<,\prec_t)$
        is isomorphic to~$(S,<,\prec)$ for some transversal~$S$ of~$P$.
        By \cref{clm:children-perm} applied on~$t'$, this permutation is in~$\Sep^{r+2}$.
        \item Similarly, applying \cref{clm:children-perm} this time to~$t$ shows that
        for any transversal~$S$ of~$\Pc$, the permutation~$(S,<,\prec_t)$ is in~$\Sep^{r+2}$.
    \end{enumerate}
    It follows from \cref{lem:grandchildren-decomp} that
    the permutation~$(A,<,\prec_t)$ is in~$\Sep^{r'}$ with
    \begin{align*}
      r' & = (r+2) + 2\left(\ceil{2 \log k} + 2\right) + 2 \\
         & \le r + 4\ceil{\log k} + 8.
    \end{align*}
    Next, we apply \cref{lem:delayed-decomp} to the delayed structured tree~$T$
    and conclude that~$(X',<,\prec)$ is in~$\Sep^{3r'}$.
    Recall that~$X'$ was one of three intervals of~$(X,\prec)$ defined by $\first$ and $\last$.
    Combining the permutations on these three intervals and on~$\{\first,\last\}$,
    we finally find that~$\sigma = (X,<,\prec)$ is a~4-shuffle of permutations in~$\Sep^{3r'}$.
    By \cref{lem:shuffle-decomp}, this implies that~$\sigma$ is in~$\Sep^{s}$ for
    \begin{align*}
      s & = 3r' + 4 \\
        & \le 3r + 12\ceil{\log k} + 28.
        \qedhere
    \end{align*}
\end{proof}

Recursively applying \cref{thm:main-induction} gives
the following bound on the length of factorisations of $k$-almost mixed free permutations.
\begin{corollary}
  \label{cor:almost-mixed-free-decomp}
  Any $k$-almost mixed free permutation factorises into a~product of at most~$4 \cdot 3^k$ separable permutations.
\end{corollary}
\begin{proof}
  \Cref{lem:2almost-mixed-free,thm:main-induction} give that $k$-almost mixed free permutations
  are product of at~most~$f(k)$ separable permutations for any function~$f$ satisfying
  \begin{align*}
    f(2) & \ge 1 \\
    \text{and } f(k) & \ge 3f(k-1) + 12\ceil{\log k} + 28.
  \end{align*}
  This is satisfied for
  \[ f(k) = 4 \cdot 3^k - 6\ceil{\log k} - 23. \]
  Indeed, we have
  \begin{align*}
    f(2) & = 4 \cdot 3^2 - 6\ceil{\log 2} - 23 = 7. \\
    \text{and } f(k) &= 4 \cdot 3^k - 6\ceil{\log k} - 23 \\
    &\ge 4 \cdot 3^k - 18(1 + \ceil{\log (k-1)}) + 12\ceil{\log k} - 23 \\
    &= 3\big(4 \cdot 3^{k-1} - 6\ceil{\log (k-1)} - 23\big) - 18 + 12\ceil{\log k} + 2 \cdot 23 \\
    &= 3f(k-1) + 12\ceil{\log k} + 28.
    \qedhere
  \end{align*}
\end{proof}

The last step is to show that pattern-avoiding permutations decompose into almost-mixed free ones.
Here, it is worth noting that permutations avoiding a~fixed pattern may contain arbitrarily large mixed minors.
This is the case of 2-shuffles: if~$<$ is the usual order on~$[n]$,
and~$\prec$ orders all odd integers before the even ones,
then the adjacency matrix of~$\prec$ ordered by~$<$ contains an $(n/2)$-mixed minor.
Thus the permutation~$\sigma = (<,\prec)$ has arbitrary mixed minors despite avoiding the pattern~$321$.
(In this specific example, the inverse permutation~$\sigma^{-1} = (\prec,<)$ \emph{is} 4-mixed free,
but one could combine~$\sigma$ and~$\sigma^{-1}$ to obtain an example
where the inverse also has arbitrary mixed minors.)

In general, we need to introduce a~third linear order to obtain a~$k$-mixed free structure.
\begin{lemma}
  \label{lem:pattern-decomp}
  For any pattern~$\pi$, there exists~$k = 2^{O(|\pi|)}$ such that
  any permutation~$\sigma$ avoiding~$\pi$ factorises as
  \[ \sigma = \sigma_2 \circ \sigma_1^{-1} \]
  where~$\sigma_1,\sigma_2$ are $k$-almost mixed free.
\end{lemma}
\begin{proof}
  Let~$(X,<_1,<_2)$ be a~biorder representing~$\sigma$.
  If~$\sigma$ avoids~$\pi$, then by \cref{thm:pattern-tww}, $(X,<_1,<_2)$ has twin-width $t = 2^{O(|\pi|)}$.
  Thus, by \cref{lem:tww-mixed-minor}, there exists a~third linear order~$<_3$ on~$X$
  such that the adjacency matrix of either~$<_1$ or~$<_2$ ordered by~$<_3$ is $(2t+2)$-mixed free.
  It follows that these adjacency matrices ordered by~$<_3$ are $k$-almost mixed free for $k = 4t + 4$.
  Defining the permutations $\sigma_1 = (X,<_3,<_1)$ and $\sigma_2 = (X,<_3,<_2)$, we obtain the result.
\end{proof}

Our main result immediately follows from \cref{cor:almost-mixed-free-decomp,lem:pattern-decomp}.
\mainthm*

\section{Algorithmic implementation}\label{sec:algo}
The proofs of the previous sections are all effective,
and yield an FPT algorithm to compute the factorisation given by \cref{thm:main}.
A~relatively naive implementation should run in quadratic time, for any fixed excluded pattern.
In this section, we explain how to improve this to get a~linear complexity
focusing on the fow sections of the proof for which this is not straightforward:
computing the delayed structured tree in \cref{lem:delayed-construction},
finding a~proper colouring of the mixed graph in \cref{lem:grandchildren-decomp},
and finding a~mixed-minor free order in \cref{lem:pattern-decomp}.

We work in the RAM model.
We assume that a~permutation~$\sigma$ is given as an array containing~$[n]$ in permuted order.
If~$\sigma$ is the biorder~$(X,<_1,<_2)$, it means that
\begin{enumerate}
    \item we can sort~$X$ according to either~$<_1$ or~$<_2$ in linear time, and
    \item after this preprocessing, we can in constant time compute
    successors and predecessors for either of the two linear orders~$<_1,<_2$
    (in addition to being able to compare two given elements).
\end{enumerate}
For example, the linear algorithm would \emph{not} apply if~$<_2$ was given as a~comparison oracle,
or as an array of arbitrary numbers.
Note that Guillemot--Marx algorithm for pattern detection~\cite{guillemot14patterns}
makes the same assumption, and we use it in the form of \cref{thm:pattern-tww}.

\subsection{Delayed substitutions}
We first describe how to compute a~delayed structured tree for a~permutation~$\sigma = (X,<,\prec)$.
We see~$\sigma$ as an array containing values from~$1$ to~$n$,
so that~$<$ is the left-to-right order on the array, and~$\prec$ compares the values.
Thus we will prefer to say `left' and `right' when referring to~$<$,
and `greater' or `smaller' when referring to~$\prec$.

We use as subroutine the classical problem of finding extrema in a~given interval of an array.
\begin{theorem}[\cite{harel1984LCA,Schieber1988LCA,bender2000LCA}]
    \label{thm:interval-max-algo}
    Given an array~$A$, one can, after a~linear-time preprocessing,
    find in constant time the minimum and maximum (and their positions) in any interval in~$A$.
\end{theorem}

Given~$x,y \in X$, let~$D(x,y)$ be the interval between~$x$ and~$y$ in the linear order~$\prec$.
This is the set of elements that \emph{distinguish}~$x$ and~$y$.

\begin{lemma}
    Given a~permutation~$\sigma$, one can compute in time~$O(\card{\sigma})$ 
    a~delayed structured tree for~$\sigma$ subject to the restrictions of \cref{lem:delayed-construction}.
\end{lemma}
\begin{proof}
    For all intents and purposes, the proof of~\cref{lem:delayed-construction} is already an algorithm.
    The issue is to implement it in linear time.
    Let us recall the important basic operation in the proof of~\cref{lem:delayed-construction}:
    given~$Y \subseteq X$ an interval for~$<$,
    we want to partition it into some intervals~$Y_1 < \dots < Y_k$ such that
    \begin{itemize}
        \item no~$z \in X \setminus Y$ distinguishes two elements of the same~$Y_i$ (i.e.~$Y_i$ is a~local module), and
        \item $Y_i$ and~$Y_{i+1}$ are distinguished by some~$z \in X \setminus Y$ (i.e.\ the~$Y_i$ are maximal).
    \end{itemize}
    The delayed structured tree is obtained by repetitively applying the former operation,
    except when~$Y$ is simultaneously an interval for~$<$ and for~$\prec$, in which case we do an arbitrary split.
    To construct the tree in linear time,
    it is sufficient to implement this splitting operation in time~$O(k)$,
    where~$k$ is the number of parts produced by the splitting.

    To this end, we define some auxiliary arrays.
    Let~$x_1,\dots,x_n$ be the elements of~$X$ ordered left to right.
    For each pair of consecutive~$x_i,x_{i+1}$,   
    let~$l_i$ and~$r_i$ be the leftmost and rightmost elements of~$D(x_i,x_{i+1})$.
    Remark that this definition is independent of~$Y$.
    Let $L$ be the array $l_1, \ldots, l_{n-1}$, and $R$, the array $r_1, \ldots, r_{n-1}$.
    \begin{claim}
        One can compute all~$l_i$ and~$r_i$ in linear time.
    \end{claim}
    \begin{claimproof}
        This is exactly \cref{thm:interval-max-algo},
        but on the inverse permutation~$\sigma^{-1}$.
        Indeed, $l_i$ and~$r_i$ are the minimum and maximum w.r.t.~$<$
        of the interval between~$x_i$ and~$x_{i+1}$ for~$\prec$.
    \end{claimproof}

    We now want to generalise this claim to any interval of~$X$.
    Given~$Y \subseteq X$ an interval for~$<$,
    let~$l(Y)$ and~$r(Y)$ be the leftmost and rightmost elements in~$X$
    that distinguish some pair of elements in~$Y$.
    Remark that if~$z$ distinguishes two elements of~$Y$,
    then it must also distinguish two \emph{consecutive} elements of~$Y$.
    It follows that if~$Y = \{x_i,\dots,x_j\}$,
    then~$l(Y)$ is the leftmost element among~$l_i,\dots,l_{j-1}$,
    and~$r(Y)$ is the rightmost element among~$r_i,\dots,r_{j-1}$.
    \begin{claim}
        \label{clm:splitting-query}
        After a~linear-time preprocessing, one can answer the following query in constant time:
        Given an interval~$Y$ of~$<$,
        \begin{enumerate}
            \item find~$l(Y)$ and~$r(Y)$, and
            \item find two consecutive~$x_i,x_{i+1} \in Y$ distinguished by~$l(Y)$, or by~$r(Y)$.
        \end{enumerate}
    \end{claim}
    \begin{claimproof}
        Apply \cref{thm:interval-max-algo} to the auxiliary arrays $L$ and~$R$.
    \end{claimproof}

    We are now ready to solve the problem of splitting an interval~$Y$ into local modules.
    Using \cref{clm:splitting-query}, we compute~$l(Y)$ and~$r(Y)$.
    If both are contained in~$Y$, then~$Y$ is simultaneously an interval of~$<$ and of~$\prec$,
    and there is nothing to do.
    Otherwise, if for instance~$l(Y) \not\in Y$,
    we also obtain from \cref{clm:splitting-query} two consecutive~$x_i,x_{i+1}$ distinguished by~$l(Y)$.
    We split~$Y$ into~$Y_1,Y_2$ containing elements left of~$x_i$, and right of~$x_{i+1}$ (inclusive).
    The above only takes constant time.
    We then recurse in~$Y_1$ and~$Y_2$.
    If for instance~$l(Y_1)$ and~$r(Y_1)$ are both inside~$Y$ (but not necessarily in~$Y_1$),
    then~$Y_1$ is a~local module of~$Y$, and we do not need to split it further.
    Otherwise, we obtain a~new position at which~$Y_1$ can be split, and we continue.
    
    This construction stops once~$Y$ has been partitioned into local modules~$Y_1,\dots,Y_k$.
    The number of steps, and hence the complexity, is proportional to the number of parts~$k$.
    Finally, each split between consecutive~$Y_i,Y_{i+1}$ is explicitly given by
    some element outside~$Y$ distinguishing the rightmost element of~$Y_i$ from the leftmost element of~$Y_{i+1}$.
    Thus the sets~$Y_i$ are also maximal as desired.
\end{proof}

\subsection{Mixed graphs}
Next, we explain how to colour the mixed graphs of \cref{lem:grandchildren-decomp} in linear time,
at the price of a~somewhat worse bound than the one given by \cref{lem:mixed-degenerate}.

\begin{lemma}
    Given a~linear order~$(X,<)$, and a~partition~$\Pc$ of~$X$
    whose mixed graph has no~$K_t$ or~$K_{t,t}$ subgraph,
    one can compute a~$4t^3$-colouring of the mixed graph of~$\Pc$
    in time $O(t^3 \card{X})$.
\end{lemma}
\begin{proof}
    We know by \cref{lem:mixed-degenerate} that the mixed graph of~$\Pc$ is $(4t^2-1)$-degenerate.
    Thus, using \cref{lem:degen-colouring}, we could compute a~colouring in linear time if we were given the mixed graph.
    The only problem is to compute the latter.
    We set $n = \card{X}$.
Recall the partition of the edges of the mixed graph used in \cref{lem:mixed-degenerate}:
    $E_1$ contains the edges~$AB$ such that the interval closures of the parts~$A,B \in \Pc$
    satisfy $\clos{A} \subseteq \clos{B}$ or $\clos{B} \subseteq \clos{A}$,
    and the remaining edges are in~$E_2$.

    The graph~$(\Pc,E_2)$ is the overlap graph of the interval closures of parts of~$\Pc$.
    By \cref{lem:circle-graph-degenerate}, its maximum edge density is at most~$2(t-1)^2$.
    In particular, the number of edges in~$E_2$ is~$O(n \cdot t^2)$,
    hence by \cref{lem:circle-graph-construction,lem:degen-colouring},
    we can compute~$(\Pc,E_2)$ and find a~$4t^2$-colouring, all in time~$O(n \cdot t^2)$.

    Let us now focus on the edges of~$E_1$.
    We will not try to compute~$E_1$ entirely.
    Instead, we consider a~colour class~$\Pc' \subseteq \Pc$ of~$(\Pc,E_2)$---%
    hence a~subset of parts satisfying that no pair~$A,B \in \Pc'$ overlaps---%
    and we will find a~$t$-colouring of the subgraph~$(\Pc',E_1)$.

    The interval closures of parts in~$\Pc'$ do not overlap,
    hence they are a~laminar family, or equivalently a~rooted forest~$F$,
    where the ancestor relation is inclusion.
    We add as leaves of~$F$ all the singletons~$\{x\}$, $x \in X$,
    whose parent is the smallest interval~$\clos{A}$ containing~$x$ with~$A \in \Pc'$, if any.
    Using a~stack and iterating over~$\Pc'$ ordered by left endpoints,
    one can compute the edges of~$F$ in time~$O(\card{X} + \card{\Pc'})$.

    Consider now an edge~$AB \in E_1$ with~$\clos{A} \subseteq \clos{B}$.
    For this edge to exist, i.e.\ for~$A$ and~$B$ to be mixed, there must be some~$b \in B \cap \clos{A}$.
    We say that~$b$ \emph{witnesses} the edge~$AB$.
    Remark that this edge may be witnessed by many different elements of~$B$.
    Now suppose that~$B$ and~$b \in B$ are fixed,
    and consider all the~$A \in \Pc'$ such that there is an edge~$AB$ witnessed by~$b$.
    These are exactly the~$A \in \Pc'$ satisfying~$b \in \clos{A} \subseteq \clos{B}$,
    that is, the nodes on the path from~$b$ to~$B$ in~$F$.
    Therefore, to compute all edges of~$E_1$ going down (in~$F$) from~$B$,
    it suffices to consider all elements~$b \in B$,
    and take the union of all paths from such a~$b$ to~$B$ in~$F$.
    It is simple to compute this in time linear in~$\card{B}$ plus the number of edges going down from~$B$.
    Repeating this process for every~$B \in \Pc'$
    allows to compute the restriction of~$E_1$ to~$\Pc'$
    in time linear in~$\card{X}$ plus the number of edges.
    Finally, we know from the proof of \cref{lem:mixed-degenerate} that~$E_1$ is $(t-1)$-degenerate.
    Thus the number of edges above is~$O(t \card{\Pc'})$,
    and we can compute a~$t$-colouring of~$(\Pc',E_1)$ in time~$O(t \card{\Pc'})$.

    Thus, we have computed a~$4t^2$-colouring of~$(\Pc,E_2)$,
    and inside each colour class~$\Pc'$, a~$t$-colouring of~$(\Pc',E_1)$.
    Combining these yields a~$4t^3$-colouring of the mixed graph~$(\Pc,E_1 \cup E_2)$.
\end{proof}

\subsection{Mixed-minor free orders}
Let us finally comment on the algorithmic aspect of \cref{lem:pattern-decomp},
since it is the part of our proof which is not self-contained.

Firstly, \cref{thm:pattern-tww}, which is a~key part of Guillemot--Marx algorithm for pattern recognition,
can be implemented efficiently:
For any fixed pattern~$\pi$, there is an algorithm which given any permutation~$\sigma$,
either finds an instance of~$\pi$ as pattern of~$\sigma$,
or finds a~partition sequence for~$\sigma$ of width~$t = 2^{O(\card{\pi})}$,
in linear time~$f(\card{\pi}) \cdot O(\card{\sigma})$.
In our case, $\sigma$ is assumed to avoid a~fixed pattern~$\pi$,
hence we obtain a~partition sequence for~$\sigma$ in linear time.

We then need to retrieve the order~$<$ given by \cref{lem:tww-mixed-minor}
for which the adjacency matrix of~$\sigma$ (as biorder) is $(2t+2)$-mixed free.
Consider a~partition sequence~$\Pc_n,\dots,\Pc_1$ for a~binary structure~$(V,R_1,\dots,R_k)$.
An ordering~$<$ of~$V$ is \emph{compatible} with this partition sequence
if for any~$i \in [n]$ and part~$X \in \Pc_i$, $X$ is an interval of~$(V,<)$.
Equivalently, one can represent the partition sequence as a~tree~$T$ with leaves~$V$,
whose internal nodes represent each merge of two parts,
and the order~$<$ is compatible with the partition sequence if and only if it is compatible with~$T$.
There always are orders compatible with any given partition sequence,
and these are the orders used in \cref{lem:tww-mixed-minor}.
That is, a~more precise statement of \cref{lem:tww-mixed-minor} is (cf.\ \cite[Theorem~5.4]{twin-width1}):
if~$\Pc_n,\dots,\Pc_1$ is a~partition sequence of width~$t$ for~$(V,R_1,\dots,R_k)$,
and~$<$ is any ordering of~$V$ compatible with it,
then the adjacency matrix of~$R_i$ ordered by~$<$ is $(2t+2)$-mixed free.
Given a~partition sequence, it is simple to compute a~compatible ordering in linear time,
by constructing the tree associated with the partition sequence, and traversing the latter.

Combining these two arguments, we obtain a~linear algorithm
to compute the decomposition given by \cref{lem:pattern-decomp}
of pattern-avoiding permutations into $k$-mixed free permutations.

\section{Factoring structures of bounded twin-width}\label{sec:graphs}
In this final section, we show how to adapt our result to graphs of bounded twin-width
in order to describe them by structures with universally bounded twin-width.
To this end, we first introduce a~representation of products of permutations as ordered graphs.

\subsection{Products of permutations as path systems}
Let~$\sigma_1,\dots,\sigma_m$ be permutations of a~linear order~$(X,<)$.
The \emph{path system representation} of the product $\sigma_m \circ \dots \circ \sigma_1$ is defined as follows.
First take~$m+1$ copies~$X_0,\dots,X_m$ of~$(X,<)$, and for each~$i \in [0,m]$ and~$x \in X$,
let~$x_i$ denote the copy of~$x$ in~$X_i$.
Then, for each~$i \in [m], x \in X$, add an edge between~$x_{i-1}$ and~$\sigma_i(x)_i$.
See \cref{fig:path-system} in the introduction for an example.
Finally, we fix a~\emph{canonical linear order} on the vertices,
namely~$X_0 < \dots < X_m$, while keeping the order inside~$X$ for each copy.

We will prove the following crucial property:
If the permutations~$\sigma_i$ have bounded twin-width,
then the path system representation has twin-width bounded independently of the length~$m$.
The arguments are exactly the same as in \cite[Proposition~6.5]{twin-width2},
we reproduce them for the sake of completeness.
Here, we use the standard sparse permutation matrices (and not the dense biorder adjacency matrices of \cref{sec:decomposition}):
the matrix~$M_\sigma$ of a~permutation~$\sigma$ has a~1 at position~$(i,\sigma(i))$ for each~$i$.
Recall that a~$k$-grid in a~0,1-matrix is a~$k$-division in which every zone contains a~1,
and that a~class~$\Cc$ of permutations avoids some pattern
if and only if there is a~$k$ such that matrices of permutations in~$\Cc$ have no $k$-grid.

\begin{lemma}
    \label{lem:grid-path-system}
    Let~$\sigma_1,\dots,\sigma_m$ be permutations whose adjacency matrices have no $r$-grid.
    Then the adjacency matrix of the path system representation of $\sigma_m \circ \dots \circ \sigma_1$,
    has no $(3r+2)$-grid when ordered according to the canonical linear order.
\end{lemma}
\begin{proof}
    Let~$G$ be the path system representing $\sigma_m \circ \dots \circ \sigma_1$,
    with vertices $X_0 \uplus \dots \uplus X_m$ as above.
    The adjacency matrix of~$G$ consists of a~double diagonal of blocks
    corresponding to the adjacency matrices of~$X_{i-1}$ against~$X_i$ for each~$i$.
    The latter is exactly the permutation matrix~$M_{\sigma_i}$, or its transpose.
    See \cref{fig:path-system-matrix}.

    \makeatletter
    \def\Ddots{\mathinner{\mkern1mu\raise\p@
    \vbox{\kern7\p@\hbox{.}}\mkern2mu
    \raise4\p@\hbox{.}\mkern2mu\raise7\p@\hbox{.}\mkern1mu}}
    \makeatother
    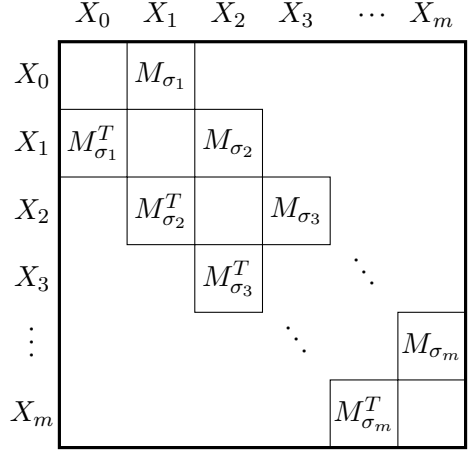
\begin{figure}[ht]
    \centering
    \begin{tikzpicture}
      \def\s{0.9}
      \def\t{0.9}
      \def\n{5}

      \begin{scope}
      \foreach \i/\j in {0/1,1/0,1/2,2/1,2/3,3/2,5/4,4/5}{
        \draw (\s*\i, -\s*\j) --+(-\s,0) --+(-\s,\s) --+(0,\s) -- cycle;
      }
      \foreach \i/\j in {0/1,1/2,2/3}{
        \node (M\i\j) at (\s*\i - \s/2, -\s*\j + \s/2) {$M_{\sigma_\j}^T$};
        \node (M\j\i) at (\s*\j - \s/2, -\s*\i + \s/2) {$M_{\sigma_\j}$};
      }

      \node (M34) at (\s*3 - \s/2, -\s*4 + 0.6*\s) {\rotatebox{90}{$\Ddots$}};
      \node (M43) at (\s*4 - \s/2, -\s*3 + 0.6*\s) {\rotatebox{90}{$\Ddots$}};
      \node (M45) at (\s*4 - \s/2, -\s*5 + \s/2) {$M_{\sigma_m}^T$};
      \node (M54) at (\s*5 - \s/2, -\s*4 + \s/2) {$M_{\sigma_m}$};

      \draw[very thick] (-\t,\t) -- (\n*\s,\t) -- (\n*\s,-\n*\s) -- (-\t,-\n*\s) -- cycle;

      \foreach \i in {0,...,3}{
        \node (R\i) at (-\t*1.4, -\i*\s + \s/2) {$X_\i$};
        \node (C\i) at (\i*\s - \s/2,\t*1.4) {$X_\i$};
      };
      \node (R4) at (-\t*1.4, -4*\s + 0.6*\s) {\rotatebox{90}{$\cdots$}};
      \node (C4) at (4*\s - \s/2,\t*1.4) {\rotatebox{90}{$\vdots$}};
      \node (R5) at (-\t*1.4, -5*\s + \s/2) {$X_m$};
      \node (C5) at (5*\s - \s/2,\t*1.4) {$X_m$};
      \end{scope}
    \end{tikzpicture}
    \caption{The adjacency matrix of the path system representation of~$\sigma_m \circ \dots \circ \sigma_1$.}
    \label{fig:path-system-matrix}
    \end{figure}

    Consider now an $l$-grid in this matrix induced by a~division~$\Rc,\Cc$.
    \begin{claim}
        \label{clm:block-intersect-bound}
        There exists~$i \in [m]$ such that every part of~$\Rc$ intersects $X_{i-1} \cup X_i \cup X_{i+1}$.
    \end{claim}
    \begin{claimproof}
        Let~$C_1$ be the first part of~$\Cc$,
        and consider~$i$ minimal such that $C_1 \subseteq X_0 \cup \dots \cup X_i$.
        Then there is no edge between~$C_1$ and~$X_j$ for~$j > i+1$,
        hence every~$R \in \Rc$ must intersect $X_0 \cup \dots \cup X_{i+1}$.
        Symmetrically, if~$C_l$ is the last part of~$\Cc$
        and~$j$ is maximal such that~$C_l \subseteq X_j \cup \dots \cup X_m$,
        we find that any~$R \in \Rc$ must intersect $X_{j-1} \cup \dots \cup X_m$.
        Thus any~$R \in \Rc$ must intersect $X_{j-1} \cup X_j \cup \dots \cup X_i \cup X_{i+1}$.
        But necessarily $i \le j$, hence any part~$R \in \Rc$ must intersect $X_{i-1} \cup X_i \cup X_{i+1}$.
    \end{claimproof}
    \begin{claim}
        \label{clm:subgrid}
        There exists~$j \in [m]$ such that at least~$\frac{l-4}{3}$ parts of~$\Rc$ are contained in~$X_j$.
    \end{claim}
    \begin{claimproof}
        By \cref{clm:block-intersect-bound},
        any part of~$\Rc$ intersects one of~$X_{i-1},X_i,X_{i+1}$.
        Excluding four parts at the borders of the latter,
        at least~$l-4$ parts of~$\Rc$ are contained in one of~$X_{i-1},X_i,X_{i+1}$.
        By the pigeonhole principle, at least~$\frac{l-4}{3}$ parts are contained in the same of these three subsets.
    \end{claimproof}

    Naturally, \cref{clm:block-intersect-bound,clm:subgrid} also hold for columns.
    Thus we obtain two sets~$X_i,X_j$ which contain at least $\frac{l-4}{3}$ parts of~$\Rc$ and~$\Cc$ respectively.
    This gives a~$\left(\frac{l-4}{3}\right)$-grid in one of the blocks of the adjacency matrix,
    which are $r$-grid free by hypothesis.
    It follows that $l \le 3r + 1$, i.e.\ the adjacency matrix is $(3r+2)$-grid free.
\end{proof}

\subsection{Subdivisions of sparse graphs}

We first consider sparse graphs.
A class~$\Cc$ of graphs is said to have \emph{bounded sparse twin-width}
if it has bounded twin-width, and it excludes a~biclique~$K_{t,t}$ as a~subgraph for some~$t$.
Bounded sparse twin-width classes were studied in \cite[section~7]{twin-width2}.
They are characterised by excluding grids in adjacency matrices.
\begin{theorem}[{\cite[Theorem~2.12]{twin-width2}}]
  Let~$G$ be a~$K_{t,t}$-subgraph-free graph of twin-width~$k$.
  Then there is some linear ordering~$<$ of the vertices of~$G$
  for which the adjacency matrix of~$G$ is $f(k,t)$-grid free, for some function~$f$.
  \label{thm:sparse-tww-grids}
\end{theorem}

\twwsubdivision*
\begin{proof}
    \newcommand{\dirE}{\vec{E}}
    Let~$G = (V,E)$ be $K_{t,t}$-subgraph-free with twin-width~$k$,
    and, applying \cref{thm:sparse-tww-grids}, consider a~linear ordering~$<$ on~$V$
    for which the adjacency matrix is $r$-grid free, with~$r$ function of~$k,t$ only.
    Fix an arbitrary orientation of the edges of~$G$, and denote by~$\dirE$ the set of oriented edges:
    for each~$uv \in E$, exactly one of~$(u,v)$ or~$(v,u)$ is in~$\dirE$.
    If~$\vec{e} = (u,v) \in \dirE$, we denote by~$s(\vec{e}) = u$ its starting point and by~$t(\vec{e}) = v$ its endpoint.

    We define two lexicographic orders on~$\dirE$:
    $<_s$ orders first by starting points (ordered by~$<$), and then by endpoints,
    while~$<_t$ orders first by endpoints, then by starting points.
    We consider the permutation~$\sigma = (\dirE,<_s,<_t)$.
    Let~$M_\sigma$ be the sparse permutation matrix of~$\sigma$:
    the columns are~$\dirE$ ordered by~$<_s$, while the rows are~$\dirE$ ordered by~$<_t$,
    and for each~$\vec{e} \in \dirE$, there is a~1 at the intersection of the row and the column corresponding to~$\vec{e}$.
    \begin{claim}
        \label{clm:halfedge-incidence-grid}
        The matrix~$M_\sigma$ is $3r$-grid free.
    \end{claim}
    \begin{claimproof}
        Let~$\Cc,\Rc$ be partitions of~$\dirE$ into intervals of~$<_s$ and~$<_t$ respectively, inducing a~$3r$-grid,
        i.e.\ for each~$C \in \Cc,R \in \Rc$, there is some~$\vec{e} \in C \cap R$.
        Consider~$\Pc_s$ the partition of~$\dirE$ which groups edges with the same starting point,
        i.e.~$\Pc_s = \{s^{-1}(v) \ : \ v \in V\}$.
        This is also a~partition of~$\dirE$ into intervals of~$<_s$.
        Remark that inside each~$P \in \Pc_s$, the orders~$<_s$ and~$<_t$ coincide,
        hence the matrix~$M_\sigma$ restricted to the columns in~$P$ has no 2-grid.
        It follows that there cannot be two distinct parts~$C,C' \in \Cc$ such that~$C,C' \subseteq P$,
        as~$\Rc,\{C,C'\}$ would yield a~2-grid using only columns of~$P$.
        Therefore, it is impossible to have more than three parts of~$\Cc$ intersecting the same part~$P \in \Pc_s$.
        Naturally, the same applies to~$\Rc$ and the partition by endpoints~$\Pc_t := \{t^{-1}(v) \ : \ v \in V\}$.

        We then pick every third part of~$\Cc$ and of~$\Rc$, yielding subsets~$\Cc' \subset \Cc$ and~$\Rc' \subset \Rc$
        such that each part~$P \in \Pc_s$ (resp.~$\Pc_t$) intersects at most one part of~$\Cc'$ (resp.~$\Rc'$).
        The families of intervals~$\Cc',\Rc'$ have size~$r$, and induce an $r$-grid in the matrix~$M_\sigma$.
        By projecting on starting points and endpoints respectively,
        we obtain~$\Cc'' := \{s(C) \ : \ C \in \Cc'\}$ and~$\Rc'' := \{t(R) \ : \ R \in \Rc'\}$,
        two families of~$r$ disjoint intervals of~$(V,<)$.
        For any~$s(C) \in \Cc''$, $t(R) \in \Rc''$,
        there is an edge~$\vec{e} \in C \cap R$
        hence~$s(\vec{e}) \in s(C)$ is adjacent to~$t(\vec{e}) \in t(R)$.
        This proves that~$\Cc'',\Rc''$ define an $r$-grid in the adjacency matrix of~$G$, a~contradiction.
    \end{claimproof}

    \Cref{clm:halfedge-incidence-grid} implies that~$\sigma$ avoids a~pattern of size~$(3r)^2$,
    hence by \cref{thm:main} we obtain a~factorisation of~$\sigma$ it into~$m = 2^{2^{O(r^2)}}$
    separable permutations,
    as $\sigma = \sigma_m \circ \dots \circ \sigma_1$.
    Consider the path system representation of this factorisation,
    with vertex set~$X_0 \uplus \dots \uplus X_m$ as in the previous section.
    The sets~$X_0,X_m$ ordered canonically are in bijection with~$(\dirE,<_s)$ and~$(\dirE,<_t)$ respectively,
    and for each~$\vec{e} \in \dirE$, there is a~path joining the copy of~$\vec{e}$ in~$X_0$
    to its copy in~$X_m$.
    We now add the vertices~$V$ of~$G$ to this structure,
    and for each edge~$\vec{e} \in \dirE$, we connect the copy of~$\vec{e}$ in~$X_0$ to~$s(\vec{e})$,
    and the copy of~$\vec{e}$ in~$X_m$ to~$t(\vec{e})$.
    Thus, for each edge~$\vec{e} \in \dirE$, a~path on~$m$ vertices joins~$s(\vec{e})$ to~$t(\vec{e})$,
    hence the resulting graph is the $(m+1)$-subdivision of~$G$.
    Furthermore, for each~$v \in V$, the neighbourhood of~$v$ inside~$X_0$, resp.~$X_m$,
    is an interval of the canonical linear order.

    To bound the twin-width of this structure, we consider the following linear order on the vertices:
    $V$ is ordered by~$<$, the path system representation of~$\sigma$ is ordered by the canonical linear order,
    and we place all vertices of~$V$ before the rest.
    The adjacency matrix of this graph then consists of
    \begin{enumerate}
        \item the adjacency matrix of the path system representation,
        which is 11-grid free using \cref{lem:grid-path-system}, because separable permutations have no 3-grid, and
        \item the adjacency matrix of~$V$ against~$X_0$ and~$X_m$,
        which consists of two increasing sequences (one for~$X_0$, one for~$X_m$),
        and can be seen to be 3-grid free.
    \end{enumerate}
    This implies that the adjacency matrix of the~$(m+1)$-subdivision of~$G$ is 15-grid free.
\end{proof}

\subsection{Transducing structures of bounded twin-width}
Finally, let us generalise the results of the former section by replacing subdivisions with first-order transductions.
We first show that permutations of twin-width~$t$ can be encoded into
structures of twin-width bounded by a~constant~$c$,
where the decoding is an FO transduction which depends on~$t$, but~$c$ is independent of~$t$.

\begin{lemma}
    \label{lem:perm-FO-desc}
    There is a~universal constant~$c$, and for any pattern~$\pi$ there is an FO interpretation~$\Phi$
    such that for any permutation~$\sigma$ avoiding~$\pi$,
    there is a~structure~$S$ of twin-width at most~$c$ such that~$\Phi(S) = \sigma$.
\end{lemma}
\begin{proof}
    The structure~$S = (V,E,<)$ is the path system representation
    of the decomposition of~$\sigma$ into separable permutations obtained by \cref{thm:main}.
    It is a~binary relational structure, with two relations:
    the edges~$E$, and the canonical order~$<$.
    By \cref{lem:grid-path-system}, since separable permutations do not contain 3-grids,
    the adjacency matrix of~$S$ has no 11-grid. This gives a~universal bound on its twin-width.

    Let us show, given the path system representation~$S$
    of the factorisation~$\sigma = \sigma_m \circ \dots \circ \sigma_1$,
    how to reconstruct the biorder~$\sigma$ using an FO interpretation.
    As previously, the sets of starting points and endpoints of paths in~$S$ are denoted by~$X_0$ and~$X_m$ respectively.
    Remark that the vertices in~$X_0 \cup X_m$ are exactly the ones with degree~1 in the edge relation~$E$,
    and that any other vertex~$x \not\in X_0 \cup X_m$ satisfies $X_0 < x < X_m$.
    This allows to test whether a~given vertex is in~$X_0$, or in~$X_m$.
    We construct a~biorder~$(X_0,<,\prec)$ isomorphic to~$\sigma$,
    where~$<$ is the already given canonical order, and~$\prec$ is defined as follows.
    Consider~$x_1,x_2 \in X_0$, and denote by~$y_i \in X_m$ the other endpoint of the path starting from~$x_i$.
    Then $x_1 \prec x_2$ if and only if $y_1 < y_2$.

    The vertex~$y_i$ is the only one in~$X_m$ connected by a~path of length~$m$ to~$x_i$.
    For fixed~$m$, `being connected by a~path of length~$m$' can be expressed by an FO formula,
    thus allowing, given~$x_1,x_2$, to identify~$y_1,y_2$.
    It follows that~$\prec$ can be defined by an FO formula,
    hence there is an FO interpretation~$\Phi$ which maps~$S$ to~$\sigma$.
    This interpretation only depends on the length~$m$ of the factorisation,
    which itself is function of the excluded pattern~$\pi$, as desired.
\end{proof}

\FOdesc*
\begin{proof}
    Firstly, the `if' part of the claim follows from the fact that the (yet to be defined) class~$\Cc$
    avoids a~pattern, hence has bounded twin-width by \cref{thm:pattern-tww},
    and transductions preserve bounded twin-width by \cref{thm:transduction-tww}.

    For the other direction, given the class~$\Dc$, we first apply \cref{thm:FO-desc2}
    to obtain a~class~$\Cc_3$ of permutations with bounded twin-width,
    and~$\Phi_3$ an FO transduction such that~$\Dc \subseteq \Phi_3(\Cc_3)$.
    Next we apply \cref{lem:perm-FO-desc}
    to obtain a~class~$\Cc_2$ of structures with twin-width bounded by a~universal constant~$c$,
    and~$\Phi_2$ an FO transduction such that~$\Cc_3 \subseteq \Phi_2(\Cc_2)$.
    Note that~$\Cc_2$ is independent of~$\Dc$.
    Finally, we apply \cref{thm:FO-desc2} a~second time
    to obtain a~class~$\Cc$ of permutations with twin-width bounded by a~universal constant~$c'$ (function of~$c$ only),
    and~$\Phi_1$ an FO transduction such that~$\Cc_2 \subseteq \Phi_1(\Cc)$.
    We conclude by composing the transductions as
    $\Cc \stackrel{\Phi_1}{\longrightarrow}
        \Cc_2 \stackrel{\Phi_2}{\longrightarrow}
        \Cc_3 \stackrel{\Phi_3}{\longrightarrow} \Dc
    $.
\end{proof}

\section*{Acknowledgments}
The authors would like to thank Patrice Ossona de Mendez for interesting discussions regarding the subject of \cref{sec:graphs}.
This work was supported by the ANR projects TWIN-WIDTH (ANR-21-CE48-0014) and Digraphs (ANR-19-CE48-0013).

\bibliographystyle{plain}

\end{document}